\documentclass[12pt,a4paper,reqno]{amsart}

\usepackage{tikz}
\usetikzlibrary{matrix,arrows,calc,snakes,patterns,decorations.markings}

\usepackage[mathscr]{eucal}
\usepackage{graphics,epic}
\usepackage{amsfonts}
\usepackage{amscd}
\usepackage{latexsym}
\usepackage{amsmath,amssymb, amsthm, stmaryrd, bm, bbm}
\usepackage[all,2cell]{xy}
\usepackage{mathrsfs}

\newtheorem{theorem}{Theorem}[section]
\newtheorem*{theorem*}{Theorem}

\newtheorem{proposition}[theorem]{Proposition}
\newtheorem{corollary}[theorem]{Corollary}

\newtheorem*{conjecture*}{Conjecture}

\theoremstyle{remark}
\newtheorem{remark}[theorem]{Remark}
\newtheorem{example}[theorem]{Example}
\theoremstyle{definition}

\newcommand{\ie}{{\em i.e.~}\ }

\newcommand{\opname}[1]{\operatorname{\mathsf{#1}}}

\renewcommand{\mod}{\opname{mod}\nolimits}

\newcommand{\proj}{\opname{proj}\nolimits}

\newcommand{\Mod}{\opname{Mod}\nolimits}

\newcommand{\add}{\opname{add}\nolimits}

\newcommand{\coh}{\opname{coh}\nolimits}

\newcommand{\cok}{\opname{cok}\nolimits}

\renewcommand{\ker}{\opname{ker}\nolimits}

\newcommand{\thick}{\opname{thick}\nolimits}

\newcommand{\per}{\opname{per}\nolimits}

\newcommand{\Hom}{\opname{Hom}}
\newcommand{\End}{\opname{End}}

\newcommand{\Ext}{\opname{Ext}}
\newcommand{\rad}{\mathrm{rad}}

\newcommand{\ten}{\otimes}

\newcommand{\Cone}{\opname{Cone}}

\newcommand{\cb}{{\mathcal B}}
\newcommand{\cc}{{\mathcal C}}
\newcommand{\cd}{{\mathcal D}}

\newcommand{\ch}{{\mathcal H}}

\newcommand{\del}{\partial}

\numberwithin{equation}{section}

\begin{document}

\title[]{Some examples of $t$-structures for finite-dimensional algebras}
\author{Dong Yang}
\address{Dong Yang, Department of Mathematics, Nanjing University, Nanjing 210093, P. R. China}
\email{yangdong@nju.edu.cn}
\dedicatory{Dedicated to Lidia Angeleri H\"ugel on the occasion of her 60th birthday}

\date{\today}

\begin{abstract}
We describe the heart of the canonical $t$-structure on the perfect derived category of a strictly positive graded algebra as the module category over the quadratic dual. Applying this result we obtain examples showing new phenomena on $t$-structures on derived categories of finite-dimensional algebras.\\
{\bf Keywords:} $t$-structure, derived category, derived representation type, piecewise hereditary algebra, quiver mutation.\\
{\bf MSC 2020:} 16E35, 16G20, 16E45.
\end{abstract}

\maketitle

\section{Introduction}
The heart of a t-structure on the derived category of a given abelian category is another abelian category sitting inside the derived category. In this paper we compare the heart with the given abelian category in the setting of finite-dimensional algebras and obtain some examples showing new phenomena. More precisely, let $A$ and $B$ be finite-dimensional algebras such that the category $\mod B$ of finite-dimensional $B$-modules is the heart of a bounded t-structure on the bounded derived category $\cd^b(\mod A)$ of $\mod A$; we show that the following situations occur:
\begin{enumerate}
\item $A$ is hereditary of type $\mathbb{D}_n$ or $\mathbb{E}_{6,7,8}$, $B$ is connected but is not derived equivalent to $A$.
\item $A$ is derived finite while $B$ is derived tame; $A$ is derived tame while $B$ is derived wild.
\item $A$ is piecewise hereditary, but $B$ is not. 
\item $A$ and $B$ are derived equivalent, but the embedding $\mod B\hookrightarrow \cd^b(\mod A)$ induced from the $t$-structure does not extend to a derived equivalence.
\end{enumerate}
Concrete examples are given in Section~\ref{s:example}. The phenomenon (4) was first observed by Martin Kalck and is recently found to occur already in Dynkin type $\mathbb{D}_5$, as a consequence of the classification of silted algebras of a hereditary algebra of type $\mathbb{D}_5$ in \cite{Xing20}. In \cite{BuanZhou18}, Buan and Zhou showed it possible that $A$ has finite global dimension while $B$ has infinite global dimension. 

\smallskip
Our examples also show that there are unexpected connections between
connected quivers of different type. It is known that connected quivers of different type are not derived equivalent, but they can be related by a bounded $t$-structure (up to derived equivalence). For example, in this way the following
pairs of quivers are related to each other
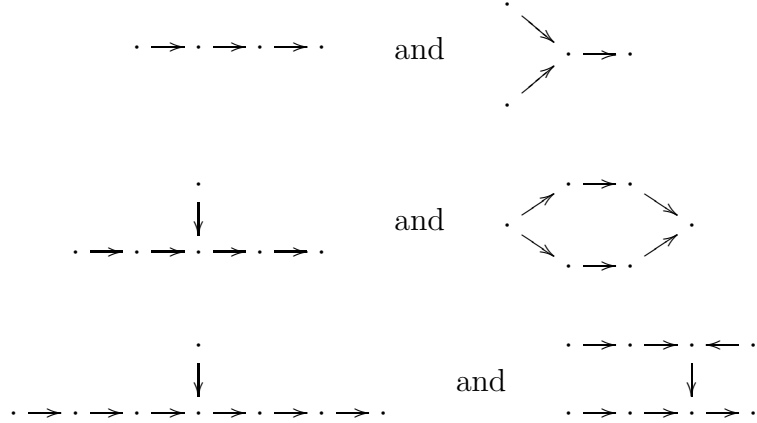
\begin{eqnarray*}
\xymatrix@R=0.5pc@C=1pc{\\ \cdot\ar[r]&\cdot\ar[r]&\cdot\ar[r]&\cdot~\\
} &\xymatrix@R=0.5pc{\\ \text{ and } \\ }&
\xymatrix@R=0.5pc@C=1pc{\cdot\ar[dr]&&\\
&\cdot\ar[r]&\cdot\\
\cdot\ar[ur] }
\end{eqnarray*}
\begin{eqnarray*}
\xymatrix@R=0.2pc@C=1pc{& &\cdot\ar[dd]&& \\ \\
\cdot\ar[r] & \cdot\ar[r] & \cdot\ar[r] & \cdot\ar[r] &\cdot~}
&\xymatrix@R=0.2pc@C=1pc{\\ \text{ and }\\ }&
 \xymatrix@R=0.2pc@C=1pc{&\cdot\ar[r]&\cdot\ar[rd]&\\
\cdot\ar[ru]\ar[rd]&&&\cdot\\
&\cdot\ar[r]&\cdot\ar[ru]&}
\end{eqnarray*}
\begin{eqnarray*}
\xymatrix@R=0.2pc@C=1pc{& &&\cdot\ar[dd]&&& \\ \\
\cdot\ar[r] & \cdot\ar[r] & \cdot\ar[r] & \cdot\ar[r]
&\cdot\ar[r]&\cdot\ar[r]&\cdot~}& \xymatrix@R=0.2pc@C=1pc{\\ \text{
and }\\ }&
\xymatrix@R=0.2pc@C=1pc{\cdot\ar[r]&\cdot\ar[r]&\cdot\ar[dd]&\cdot\ar[l]\\
\\ \cdot\ar[r]&\cdot\ar[r]&\cdot\ar[r]&\cdot}
\end{eqnarray*}

The construction of all examples is based on the following result. 

\begin{theorem*}[{\bf \ref{t:heart-2}}]
Let $Q$ be a finite quiver which is gradable in the sense of \cite[Section 1.1]{BautistaLiu17}, and let $A$ be a quotient algebra of the path algebra $kQ$. Then $\cd^b(\mod A)$ admits a bounded $t$-structure whose heart is equivalent to the category of finite-dimensional (non graded) modules over the quadratic dual $A^!$ of $A$. If all the relations of $A$ are of length at least $3$, then this heart is derived equivalent to $kQ$.
\end{theorem*}

The definition of quadratic dual is given in Section~\ref{s:heart}. We remark that in our definition the quadratic dual is a non graded algebra.
Theorem~\ref{t:heart-2} is a consequence of the following more general theorem.

\begin{theorem*}[\bf \ref{t:heart}]
Let $A=kQ/(R)$ be a graded algebra, where $Q$ is a finite graded quiver with all arrows in positive degrees and $R$ consists of homogeneous elements of the path algebra $kQ$ contained in the square of the radical of $kQ$. Then the perfect derived category of $A$, viewed as a dg algebra with trivial differential, admits a bounded $t$-structure whose heart is equivalent to the category of finite-dimensional (non graded) modules over the quadratic dual $A^!$.
\end{theorem*}

In Section~\ref{s:preliminary} we recall the basics on derived categories of dg algebras, Koszul duality, silting objects, simple-minded collections and $t$-structures.
In Section~\ref{s:ADE} we extend the ADE chain and obtain piecewise hereditary algebras of tree type $\mathbb{T}_{2,\ldots,2,n}$. In Section~\ref{s:heart} we prove the main results and in Section~\ref{s:example} we provide concrete examples.

\smallskip
Throughout, let $k$ be a field.

\subsection*{Acknowledgement.} The author thanks Jorge Vit\'oria, Adam-Christiaan van Roosmalen and Martin Kalck for stimulating discussions. He thanks Adam-Christiaan van Roosmalen for pointing out an error in an old version of Theorem~\ref{t:heart-2}. He acknowledges support by the National Natural Science Foundation of China No. 11671207 and the DFG program SPP 1388 (YA297/1-1 and KO1281/9-1).

\section{Preliminaries}\label{s:preliminary}

In this section, we recall some basics on quivers, derived categories of dg algebras, Koszul duality, silting objects, simple-minded collections and $t$-structures.

\subsection{Quivers}
\label{ss:quivers}

Let $Q$ be a finite quiver. We denote its vertex set by $Q_0$ and its arrow set by $Q_1$. 
For $\alpha\in Q_1$, denote by $s(\alpha)$ and $t(\alpha)$ the source and the target of $\alpha$, respectively. For each $i\in Q_0$, there is a \emph{trivial path} $e_i$ with $s(e_i)=t(e_i)=i$.  For $\alpha\in Q_1$, we introduce its formal inverse $\alpha^{-1}$ with $s(\alpha^{-1})=t(\alpha)$ and  $t(\alpha^{-1})=s(\alpha)$. 
A \emph{walk} is a formal product $w=c_m\cdots c_1$, where each $c_l$ is either a trivial path, or an arrow, or the inverse of an arrow such that $t(c_l)=s(c_{l+1})$ for $1\leq l\leq m-1$. We say that $w$ is a \emph{closed walk} if $s(c_1)=t(c_m)$. The \emph{virtual degree} of $w$ is defined as $d(w)=\sum_{l=1}^m d(c_l)$, where for $\alpha\in Q_1$ we define $d(\alpha)=1$ and $d(\alpha^{-1})=-1$ and for $i\in Q_0$ we define $d(e_i)=0$. Following \cite[Section 1.1]{BautistaLiu17}, we say that $Q$ is \emph{gradable} if every closed walk is of virtual degree $0$. If $Q$ is gradable, then $Q$ has no oriented cycles. The converse is not true, but if the underlying graph of $Q$ has no cycles, then $Q$ is gradable.

A \emph{non-trivial path} of $Q$ is a walk $\alpha_m\cdots\alpha_1$, where all $\alpha_l$ are arrows. 
The \emph{path algebra} $kQ$ of $Q$ is the $k$-algebra which has all paths of $Q$ (including the trivial paths) as basis and whose multiplication is given by concatenation of paths. If $Q$ is a graded quiver, then $kQ$ is naturally a graded $k$-algebra. The \emph{complete path algebra} $\widehat{kQ}$ of $Q$ is the completion of $kQ$ with respect to the $J$-adic topology, where $J$ is the ideal of $kQ$ generated by all arrows. Let $\widehat{J}$ be the ideal of $\widehat{kQ}$ generated by all arrows. For an ideal $I$ of $\widehat{kQ}$ let $\overline{I}$ denote the closure of $I$ with respect to the $\widehat{J}$-adic topology.

\subsection{Derived categories of dg algebras}

Let $A$ be a dg $k$-algebra, \ie a graded $k$-algebra equipped with the structure of a complex with differential $d$ such that the following graded Leibniz rule holds
\[d(a b)=d(a)b+(-1)^{|a|}ad(b),\]
where $a$ is homogeneous of degree $|a|$. We denote by $\cd(A)$ the derived category of (right) dg $A$-modules. It is a triangulated category with suspension functor $\Sigma$ being the shift $[1]$ of complexes. See~\cite{Keller94,Keller06d}. For an object $M$ of $\cd(A)$, denote by $\thick(M)$ the thick subcategory of $\cd(A)$ generated by $M$, \ie the smallest triangulated subcategory of $\cd(A)$ which contains $M$ and which is closed under taking direct summands. Let $\per(A):=\thick(A_A)$ denote the thick subcategory of $\cd(A)$ generated by the free dg $A$-module of rank $1$ and let $\cd_{fd}(A)$ denote the full (triangulated) subcategory of $\cd(A)$ consisting of those dg $A$-modules with finite-dimensional total cohomology. In the case that $A$ is a finite-dimensional $k$-algebra (viewed as a dg algebra concentrated in degree $0$), we have $\cd(A)=\cd(\Mod A)$, $\per(A)\simeq \ch^b(\proj A)$ and $\cd_{fd}(A)\simeq\cd^b(\mod A)$. Here for a $k$-algebra $A$, we denote by $\mod A$ the category of finite-dimensional (right) $A$-modules, by $\proj A$ the category of finitely generated projective $A$-modules, by $\ch^b(\proj A)$ the bounded homotopy category of $\proj A$ and by $\cd^b(\mod A)$ the bounded derived category of $\mod A$.

\subsection{Silting objects, simple-minded collections and $t$-structures}\label{ss:t-str}
Let  
$\cc$ be a Krull--Schmidt triangulated $k$-category with suspension functor $\Sigma$.
 
An object $M$ of $\cc$ is called a \emph{silting object} if 
\begin{itemize}
\item[$\cdot$]
$\Hom(M,\Sigma^p M)=0$ for $p>0$, 
\item[$\cdot$]
$\cc=\thick(M)$.
\end{itemize} 
A collection of objects $\{X_1,\ldots,X_n\}$ of $\cc$ is called a \emph{simple-minded collection} if 
\begin{itemize}
\item[$\cdot$] $\Hom(X_i,\Sigma^p X_j)=0$, for $p<0$ and $1\leq i,j\leq n$,
\item[$\cdot$] $\End(X_i)$ is a division algebra and $\Hom(X_i,X_j)=0$ for $1\leq i\neq j\leq n$,
\item[$\cdot$] $\cc=\thick(X_1,\ldots,X_r)$.
\end{itemize}
A \emph{$t$-structure} on $\cc$ is a pair $(\cc^{\leq
0},\cc^{\geq 0})$ of strictly full (that is, closed under
isomorphisms) subcategories of $\cc$ such that
\begin{itemize}
\item[$\cdot$] $\Sigma\cc^{\leq 0}\subseteq\cc^{\leq 0}$ and
$\Sigma^{-1}\cc^{\geq 0}\subseteq\cc^{\geq 0}$;
\item[$\cdot$] $\Hom(M,\Sigma^{-1}N)=0$ for $M\in\cc^{\leq 0}$
and $N\in\cc^{\geq 0}$,
\item[$\cdot$] for any $M\in\cc$ there is a triangle $M'\rightarrow
M\rightarrow M''\rightarrow\Sigma M'$ in $\cc$ with $M'\in\cc^{\leq
0}$ and $M''\in\Sigma^{-1}\cc^{\geq 0}$.
\end{itemize}
The \emph{heart} $\ch=\cc^{\leq 0}~\cap~\cc^{\geq 0}$ is
always abelian. The
$t$-structure $(\cc^{\leq 0},\cc^{\geq 0})$ is said to  be
\emph{bounded} if $\cc=\thick(\ch)$ and \emph{algebraic} if in addition all objects of its heart $\ch$ have finite length.

\begin{theorem}
\label{t:koenigy}
\emph{(\cite[Theorem 6.1, Lemma 5.2, Lemma 5.3(a) and Proposition 5.4]{KoenigYang14})}
Let $A$ be a finite-dimensional $k$-algebra.
There are one-to-one correspondences among the following sets:
\begin{itemize}
\item[(i)] basic silting objects of $\ch^b(\proj A)$, up to isomorphism, 
\item[(ii)] simple-minded collections of $\cd^b(\mod A)$, up to isomorphism, 
\item[(iii)] algebraic $t$-structures on $\cd^b(\mod A)$.
\end{itemize}

Under the above correspondences, a basic silting object $M=M_1\oplus\ldots\oplus M_n$ and the corresponding simple-minded collection $\{X_1,\ldots,X_n\}$ determine each other (up to reordering) by the following Hom-duality
\[\begin{array}{ll}
& \Hom(M_i,\Sigma^p X_j)=0 \text{ if } p\neq 0 \text{ or } i\neq j\\
&\Hom(M_i,X_i) \text{ is free of rank $1$ over } \End(X_i);
\end{array}\]
the heart of an algebraic $t$-structure is the extension closure of the corresponding simple-minded collection and is equivalent to the module category over the endomorphism algebra of the corresponding silting object.
\end{theorem}

For example, the following correspond to each other for a basic finite-dimensional $k$-algebra $A$:
\begin{itemize}
\item[-] the free $A$-module $A_A$ of rank $1$,
\item[-] a complete set of pairwise non-isomorphic simple $A$-modules,
\item[-] $(\cd^{\leq 0}_{\mathrm{std}},\cd^{\geq 0}_{\mathrm{std}})$, where $\cd^{\leq 0}_{\mathrm{std}}$ (respectively, $\cd^{\geq 0}_{\mathrm{std}}$) is the full subcategory of $\cd^b(\mod A)$ consisting of complexes with trivial cohomologies in positive degrees (respectively, in negative degrees).
\end{itemize}

\section{Extending the ADE chain}\label{s:ADE}

A finite-dimensional $k$-algebra $A$ is said to be \emph{piecewise hereditary}~\cite{HappelReitenSmalo96} if there
is a hereditary abelian category $\ch$ such that $\cd^b(\mod A)$ is triangle equivalent to $\cd^b(\ch)$. The category $\ch$ is called the \emph{type} of $A$. This
is well defined up to derived equivalence.
Typical examples of such $\ch$ are $\mod H$ for
some finite-dimensional hereditary algebra $H$ and the category
of coherent sheaves $\coh \mathbb{X}$ for the weighted projective line
$\mathbb{X}=\mathbb{X}_{p_1,\ldots,p_r}$ (defined by Geigle and Lenzing~\cite{GeigleLenzing87}) with
some weights $(p_1,\ldots,p_r)$. In this case, the algebra $H$ or the weighted
projective line $\mathbb{X}$ is also called the type of $A$. 
\smallskip

Consider the algebra $\Lambda(n,r)$ ($n> r\geq 3$) given by the linear quiver of type $\mathbb{A}_n$
\[\xymatrix@C=1.5pc{1\ar[r]&2\ar[r]&\cdot\ar@{.}[r]&\cdot\ar[r]&n}\]
with relations all paths of length $r$.

For $n=1,\ldots,11$, the algebras $\Lambda(n,3)$ are piecewise
hereditary of type $\mathbb{A}_1,\mathbb{A}_2,\mathbb{A}_3$,
$\mathbb{D}_4,\mathbb{D}_5,\mathbb{E}_6,\mathbb{E}_7,\mathbb{E}_8,\tilde{\mathbb{E}}_8,\mathbb{X}_{236},\mathbb{X}_{237}$, respectively.
Notice that as $n$ increases, although the representation theory of
$\Lambda(n,3)$ does not change much (they are all of finite
representation type), the complexity of their derived categories
increases drastically ($\Lambda(n,3)$ is derived wild for $n\geq
11$). In~\cite{HappelSeidel10}, Happel and Seidel provide a complete
list of the pairs $(n,r)$ such that $\Lambda(n,r)$ is piecewise
hereditary and give the type:
\begin{itemize}
 \item[$\cdot$] the first family: $(n,r)=(n,n-1)$ (of type $\mathbb{D}_n$),
 \item[$\cdot$] the second family: $(n,r)=(n,n-2)$ (of type $\mathbb{T}_{23(n-3)}$),
 \item[$\cdot$] the third family: $(n,r)=(n,n-3)$ (of type $\mathbb{T}_{23(n-3)}$),
 \item[$\cdot$] the fourth family: $(n,r)=(n,n-4)$ (of type $\mathbb{X}_{23(n-4)}$),
 \item[$\cdot$] exceptional cases: $(n,r)=(8,3)$ (of type $\mathbb{E}_8=\mathbb{T}_{235}$),
 $(9,3)$ (of type $\tilde{\mathbb{E}}_8=\mathbb{T}_{236}$), $(10,3)$ (of type $\mathbb{X}_{236}$),
$(11,3)$ (of type $\mathbb{X}_{237}$), $(9,4)$ (of type
$\mathbb{X}_{244}$), $(10,4)$ (of type $\mathbb{X}_{245}$), $(10,5)$
(of type $\mathbb{T}_{237}$) and $(11,6)$ (of type
$\mathbb{X}_{237}$).
\end{itemize}
Here for a sequence of integers $(p_1,\ldots,p_r)$,
$\mathbb{T}_{p_1,\ldots,p_r}$ is the star quiver with $r$ arms which have
$p_1-1,\ldots,p_r-1$ arrows, respectively
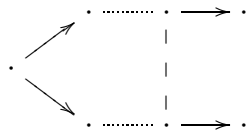
\[\xymatrix@R=0.7pc@C=1.5pc{&\cdot\ar@{.}[r]&\cdot\ar[r]\ar@{--}[dd]&\cdot\\
\cdot\ar[ur]\ar[dr]&&\\
&\cdot\ar@{.}[r]&\cdot\ar[r]&\cdot}\]
Since this is a tree, it is derived equivalent to any quiver with the same underlying graph.

\medskip
We extend this result a little. For $m\geq 2$ and $n\geq 3$, consider the following quiver
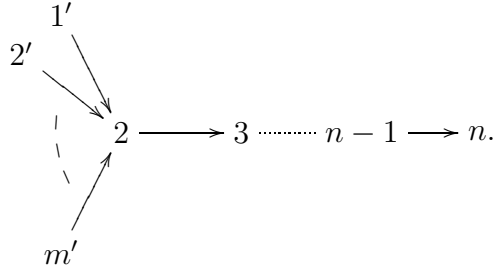
\[
\begin{xy}0;<0.75pt,0pt>:<0pt,-0.75pt>::
(15,0) *+{1'}="1",
(-5,20) *+{2'}="2",
(15,120) *+{m'}="3",
(45,60) *+{2}="4",
(105,60) *+{3}="5",
(165,60) *+{n-1}="6",
(225,60) *+{n.}="7",
(15,40) *+{}="8",
(25,95) *+{}="9",
"1", {\ar "4"}, "2", {\ar"4"}, "3", {\ar "4"},
"4", {\ar "5"}, "5", {\ar@{.}"6"}, "6", {\ar"7"},
"8", {\ar@/_5pt/@{--} "9"},
\end{xy}
\]
Let $\Lambda'_m(n,n-1)$ be the algebra given by this quiver with relations all paths of length $n-1$. The main result of this section is

\begin{proposition}\label{p:extend-ADE}
 The algebra $\Lambda'_m(n,n-1)$ is piecewise hereditary of type $\mathbb{T}_{2^{m+1},n-2}$.
\end{proposition}

\begin{proof} Let $A=\Lambda'_m(n,n-1)$. 
Let $P_n$ be the indecomposable projective $A$-module corresponding to the vertex $n$. Consider $T=D(A/Ae_nA)\oplus P_n$. Because $A/Ae_nA$ is hereditary, there exists a short exact sequence of $A/Ae_nA$-modules
\[
\xymatrix{0\ar[r] & A/Ae_nA\ar[r] & I^0 \ar[r] & I^1\ar[r] & 0}
\]
with $I^0,I^1\in\add(D(A/Ae_nA))$, which induces a short exact sequence of $A$-modules
\[
\xymatrix{0\ar[r] & A\ar[r] & I^0\oplus P_n \ar[r] & I^1\ar[r] & 0.}
\]
Moreover, since $P_n$ is projective-injective, it follows that $T$ has no self-extensions. So $T$ is a tilting $A$-module. It is straightforward to check that the endomorphism algebra of $T$ is the path algebra of
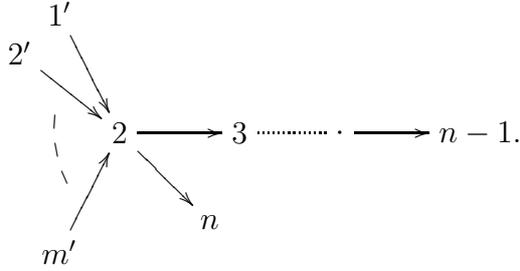
\[
\begin{xy}0;<0.75pt,0pt>:<0pt,-0.75pt>::
(15,0) *+{1'}="1",
(-5,20) *+{2'}="2",
(15,120) *+{m'}="3",
(45,60) *+{2}="4",
(105,60) *+{3}="5",
(155,60) *+{\cdot}="6",
(225,60) *+{n-1.}="10",
(90,105) *+{n}="7",
(15,40) *+{}="8",
(25,95) *+{}="9",
"1", {\ar "4"}, "2", {\ar"4"}, "3", {\ar "4"},
"4", {\ar "5"}, "5", {\ar@{.}"6"}, "4", {\ar"7"},
"8", {\ar@/_5pt/@{--} "9"}, "6", {\ar "10"},
\end{xy}
\]
\end{proof}

An alternative proof of Proposition~\ref{p:extend-ADE} using quiver mutation is given in Appendix~\ref{a:quiver-mutation}.

\section{Quadratic hearts}\label{s:heart}

Let $Q$ be a positively graded finite quiver, that is, a finite graded quiver with all arrows in positive degrees, and $A= kQ/(R)$ with $R$ a set of minimal homogeneous relations. In this section we show that the heart of the canonical $t$-structure on $\per(A)$ described in \cite[Theorem 16]{Schnuerer11} and \cite[Theorem 7.1]{KellerNicolas13} is equivalent to the category of finite-dimensional modules over the quadratic dual of $A$.

\medskip
Write $R=R_2\cup R_{\geq 3}$, where $R_2$ (respectively, $R_{\geq 3}$) consists of relations in $R$ of degree $2$ (respectively, at least $3$).
Denote by $Q^1$ the degree $1$ component of $Q$, \ie the vertices of $Q^1$ are the same as $Q$ and the arrows of $Q^1$ are those arrows of $Q$ of degree $1$, and consider it as a quiver concentrated in degree $0$. The \emph{quadratic dual} of $A$ is defined as the algebra $A^!:=\widehat{k(Q^1)^{op}}/\overline{(R_2^{\perp})}$, the quotient of the complete path algebra of the opposite quiver $(Q^1)^{op}$ of $Q^1$ by the closure of the ideal generated by $R_2^{\perp}$ (for the definition of complete path algebras see Section~\ref{ss:quivers}). Here we consider the arrow span of the opposite quiver $(Q^1)^{op}$ as the dual space $DV:=\Hom_K(V,K)$ of the arrow span $V$ of $Q^1$, consider $R_2$ as a subset of $V\ten_K V$ and $R_2^{\perp}$ is the orthogonal to $R_2\subseteq V\ten_K V$ with respect to the standard bilinear form $(V\ten_K V)\times (DV\ten_K DV)\rightarrow k$. We remark that $A^!$ is a non graded algebra\footnote{If all arrows of $Q$ are in degree $1$ and $R=R_2$, then in the literature the quadratic dual of $A$ is defined as the graded algebra $kQ^{op}/(R^{\perp})$, where all arrows of $Q^{op}$ are put in degree $1$, see for example \cite{BeilinsonGinzburgSoergel96}. We recover this classical construction from ours by adding an extra grading on $Q$, namely, by putting any arrow $\alpha$ of $Q$ in bidegree $(-1,|\alpha|)$.}.

\smallskip

Let $\cd^{\leq 0}$ (respectively, $\cd^{\geq 0}$) be the smallest full subcategory of $\per(A)$ containing $A$ and closed under extensions, shift $\Sigma$ (respectively, negative shift $\Sigma^{-1}$) and direct summands. We thank Olaf M. Schn\"urer for pointing out that the following theorem generalises \cite[Theorem 39]{Schnuerer11}.

\begin{theorem}\label{t:heart}
The pair $(\cd^{\leq 0},\cd^{\geq 0})$ is a bounded $t$-structure on $\per(A)$  whose heart is equivalent to $\mod A^!$.
\end{theorem}
\begin{proof}
That $(\cd^{\leq 0},\cd^{\geq 0})$ is a bounded $t$-structure on $\per(A)$ follows from \cite[Theorem 16]{Schnuerer11} and \cite[Theorem 7.1]{KellerNicolas13}. Next we show that the heart is equivalent to $\mod A^!$.

Let $\bar{A}=\bigoplus_{p\geq 1}A^p$ and $K=kQ_0$. Then the multiplication of $A$ defines a graded $K$-bimodule structure on $\bar{A}$.
Consider the tensor algebra
\[T_K(\bar{A}[1])=K\oplus\bar{A}[1]\oplus \bar{A}[1]\ten_K \bar{A}[1]\oplus\ldots=\bigoplus_{p\geq 0} (\bar{A}[1])^{\ten_K p}.\]
The Koszul dual $A^*$ of $A$ is defined as
\begin{eqnarray}A^*:=\Hom_K(T_K (\bar{A}[1]),K)=\widehat{T}_K(\Hom_K(\bar{A}[1],K)),\label{construction:koszul-dual}\end{eqnarray}
where $\widehat{T}$ denotes the complete tensor algebra. 
This is a dg algebra whose multiplication is given by the tensor product and whose differential $d$ is induced by the multiplication of $\bar{A}$. Precisely, for an element $f$ in $\Hom_K(\bar{A}[1],K)$ and two homogeneous elements $a$ and $b$ in $\bar{A}[1]$, we have
\begin{eqnarray}d(f)(a\ten b)=f((-1)^{|a|}ab).\label{eq:differential}\end{eqnarray}

Observe that the degree $0$ component of $A^*$ is the complete path algebra $\widehat{k(Q^1)^{op}}$, and the intersection of the image of the differential $d$ of $A^*$ and the degree $0$ component is the closure of the ideal of $\widehat{k(Q^1)^{op}}$ generated by $d(f)$ for $f\in D(A^2[1])$, where $A^2$ is the degree $2$ component of $A$. Denote by $V^2$ the $k$-span of the arrows of $Q$ of degree $2$ and by abuse of notation denote by $R_2$ the $k$-span of $R_2$. Then $A^2=V^2\oplus (V\ten_K V)/R_2$. The formula (\ref{eq:differential}) implies that $d(f)=0$ for $f\in D(V^2)$ and for $f\in D((V\ten_K V)/R_2)$ the differential $d(f)$ equals the image of $f$ under the canonical isomorphism $D((V\ten_K V)/R_2)\cong R_2^\perp$. Therefore $H^0(A^*)=A^!$.

Let $\mathbf{p}S$ be the bar resolution of $A/\rad A$ and let $B$ be the dg endomorphism algebra of $\mathbf{p}S$. Then by \cite[Lemma 3.7]{SuHao18}, there is a quasi-isomorphism of dg algebras from $A^*$ to $B$. Going through the proof of \cite[Theorem 7.1]{KellerNicolas13}, we see that the heart of $(\cd^{\leq 0},\cd^{\geq 0})$ is equivalent to $\mod H^0(B)$, which is equivalent to $\mod H^0(A^*)=\mod A^!$. 
\end{proof}

\begin{remark}
\begin{itemize}
\item[(a)]
By the general result \cite[Proposition 3.1.10]{BeilinsonBernsteinDeligne82}, there is a triangle functor
\[
\mathrm{real}\colon \cd^b(\mod A^!)\rightarrow \per(A),
\]
which extends the equivalence between $\mod A^!$ and the heart.
\item[(b)]
If $A$ has no relations of degree $2$, then by definition any path of length $2$ is a relation of the quadratic dual $A^!$, and hence $A^!$ has radical square zero. See \cite{BautistaLiu06,BautistaLiu17,BekkertDrozd09,Yang18} for studies on derived categories of algebras with radical square zero.
\end{itemize}
\end{remark}

Let $Q$ be a gradable finite quiver (Section~\ref{ss:quivers}), $I$ an ideal of $kQ$ generated by relations of length at least $2$ and let $A=kQ/I$. Then $A$ can be graded such that all arrows of $Q$ are of degree $1$. Denote this graded algebra by $A^{gr}$ and the graded quiver by $Q^{gr}$. Let $A^!:=(A^{gr})^!$, which we call the \emph{quadratic dual} of $A$ by abuse of language.

Notice that under the assumption on $Q$ the algebra $A$ is finite-dimensional and has finite global dimension. Thus $\cd^b(\mod A)=\per(A)\simeq\ch^b(\proj A)$. Moreover, there is a set of integers
$\{t_i\mid i\in Q_0\}$ such that if there is a path from $i$ to $j$ of length $l$ then $t_i-t_j=l$. We fix such a set. For $i\in Q_0$, put $P_i=e_iA$, where $e_i$ is the trivial path at $i$. Let $\cd'^{\leq 0}$ (respectively, $\cd'^{\geq 0}$) be the smallest full subcategory of $\cd^b(\mod A)$ containing $\Sigma^{t_i}P_i$, $i\in Q_0$, and closed under extension, shift $\Sigma$ (respectively, negative shift $\Sigma^{-1}$) and direct summands.

\begin{theorem}\label{t:heart-2} Keep the above notation and assumptions. Then $(\cd'^{\leq 0},\cd'^{\geq 0})$ is a bounded $t$-structure on $\cd^b(\mod A)$ whose heart is equivalent to $\mod A^!$. If all relations in $I$ have length at least $3$, then $A^!$ has radical square zero and the heart is derived equivalent to $kQ$.
\end{theorem}
\begin{proof}
For $i\in Q_0$, put $P_i^{gr}=e_iA^{gr}$. Then $\bigoplus_{i\in Q_0}\Sigma^{-t_i}P_i^{gr}$ is a tilting object of $\per(A^{gr})$ whose endomorphism algebra is isomorphic to $A$.
Therefore $\per(A)=\cd^b(\mod A)$ is triangle equivalent to $\per(A^{gr})$. Under this equivalence the pair $(\cd'^{\leq 0},\cd'^{\geq 0})$ is sent to the bounded $t$-structure $(\cd^{\leq 0},\cd^{\geq 0})$ in Theorem~\ref{t:heart} applied to $A^{gr}$. We obtain the first statement.

If all relations in $I$ have length at least $3$, then $A^!$ is the radical square zero algebra given by the quiver $Q^{op}$. By \cite[Corollary 3.14]{BautistaLiu06}, $A^!$ is derived equivalent to $kQ$. 
\end{proof}

\subsection{An alternative proof of the second statement of Theorem~\ref{t:heart-2}}

The second statement of Theorem~\ref{t:heart-2} is the result we will mainly use to construct $t$-structures in the next section. Here we state it separately and give an alternative proof. Moreover, we give a necessary condition under which the embedding from the heart extends to a derived equivalence.

Let $Q$ be a finite gradable quiver and $A=kQ/I$ with $I$ consisting of relations of length at least $3$. Let $\{t_i\mid i\in Q_0\}$ and $(\cd'^{\leq 0},\cd'^{\geq 0})$ be as in the preceding subsection.

\begin{theorem}\label{t:hereditary-heart}
The pair $(\cd'^{\leq 0},\cd'^{\geq 0})$ is a bounded $t$-structure on $\cd^b(\mod A)$
whose heart $\cb$ is derived equivalent to $kQ$. Moreover, the embedding $\cb\to\cd^b(\mod A)$ extends to a derived equivalence $\cd^b(\cb)\to\cd^b(\mod A)$ unless $I=0$.
\end{theorem}

\begin{proof}
We will find a silting object in $\cd^b(\mod A)\simeq \ch^b(\proj A)$ whose endomorphism algebra
is isomorphic to $kQ^{op}/(arrows)^2$ and show that this radical square zero algebra is derived equivalent to $kQ$. By applying Theorem~\ref{t:koenigy},
we obtain the first statement.

\smallskip

The collection
$\{\Sigma^{t_1}P_1,\ldots,\Sigma^{t_n}P_n\}$ is simple-minded.
Indeed, the graded endomorphism algebra
$$\bigoplus_{m\in\mathbb{Z}}\Hom(\bigoplus_{i\in Q_0}\Sigma^{t_i}P_i,\Sigma^m
\bigoplus_{i\in Q_0}\Sigma^{t_i}P_i)$$ 
is isomorphic to $A^{gr}$,  the graded
algebra $A=kQ/I$ with each arrow in degree $1$. The pair $(\cd'^{\leq 0},\cd'^{\geq 0})$ is exactly the bounded $t$-structure corresponding to this simple-minded collection. 
By the Hom-duality
(Theorem~\ref{t:koenigy}),
the silting object corresponding to this simple-minded collection is
$M=\bigoplus_{i\in Q_0}\Sigma^{t_i}\nu^{-1}S_i$, where $\nu$ is the derived Nakayama functor and $S_i$ is the simple top of $P_i$ ($i\in Q_0$).  Therefore the heart of $(\cd'^{\leq 0},\cd'^{\geq 0})$ is equivalent to $\mod\End(M)$.

Next we compute the endomorphism algebra of $M$, which is isomorphic to the endomorphism algebra of $\nu(M)=\bigoplus_{i\in Q_0}\Sigma^{t_i}S_i$ because in this case $\nu$ is a triangle equivalence. For $i,j\in Q_0$,
there are isomorphisms
\[\Hom(\Sigma^{t_i}S_i,\Sigma^{t_j}S_j)=\Hom(S_i,\Sigma^{t_j-t_i}S_j)=\Ext^{t_j-t_i}_A(S_i,S_j).\]
If there is an arrow from $j$ to $i$ in $Q$, then $t_j-t_i=1$. In this case,
the dimension of $\Ext^1_\Lambda(S_i,S_j)$ is the number of arrows from $j$ to $i$.
Suppose $t_j-t_i\geq 2$ and $\Ext_\Lambda^{t_j-t_i}(S_i,S_j)\neq 0$. Then it follows from~\cite[Theorem 1.1]{Bongartz83} that there is a path from $j$ to $i$ of length
at least the integer part of $\frac{3(t_j-t_i)-1}{2}$, since $I\subseteq J^3$, where $J$ is the ideal of $kQ$ generated by all arrows. This contradicts the fact that a path from $j$ to $i$ is of length $t_j-t_i$.
Therefore, if $t_j-t_i\geq 2$ then $\Ext_\Lambda^{t_j-t_i}(S_i,S_j)= 0$, and hence
$\End(M)\cong kQ^{op}/(arrows)^2$.

Finally, we claim that $M$ is a tilting object (equivalently, $\nu(M)$ is a tilting object) if and only if $I=0$. This implies that $kQ$ is derived equivalent to $kQ^{op}/(arrows)^2$.
Suppose $I=0$. Then $A$ is hereditary, and hence
$\Ext^m_A(S_i,S_j)=0$ for any $m\geq 2$ and for any $i,j\in Q_0$. Moreover, $\Ext^1_A(S_i,S_j)\neq 0$ if and only if there are
arrows from $j$ to $i$ in $Q$. Therefore
\[\Hom(\Sigma^{t_i}S_i,\Sigma^m\Sigma^{t_j}S_j)=\Ext_A^{m+t_j-t_i}(S_i,S_j)=0\]
unless $m=0$.
Conversely, suppose $I\neq 0$, and there is a minimal relation, say, from $j$ to $i$ of length $l$ ($l\geq 3$ by the assumption on $I$).
Then $t_j-t_i=l$, and
\[\Hom(\Sigma^{t_i}S_i,\Sigma^{2-l}\Sigma^{t_j}S_j)=\Ext_A^{2-l+t_j-t_i}(S_i,S_j)=\Ext^2_A(S_i,S_j),\]
which does not vanish by~\cite[Theorem 1.1]{Bongartz83}.
Therefore $M$ is not a tilting object.

Now assume $I\neq 0$. It is known that the Yoneda algebra of $kQ^{op}/(arrows)^2$ is $kQ^{gr}$, where $Q^{gr}$ is the graded quiver obtained from $Q$ by putting all arrows in degree $1$ (this can be obtained by direct computation or by general theory \cite[Theorem 2.10.1]{BeilinsonGinzburgSoergel96}. Recall from the second paragraph above that $\Sigma^{t_1}P_1,\ldots,\Sigma^{t_n}P_n$ is a complete set of pairwise non-isomorphic simple objects of $\cb$ and the corresponding Yoneda algebra in $\cd^b(\mod A)$ is $A^{gr}$. Therefore the embedding $\cb\to\cd^b(\mod A)$ cannot be extended to a triangle equivalence.
\end{proof}

\section{Examples}\label{s:example}

Let $A$ be a finite-dimensional $k$-algebra and $(\cd^{\leq 0},\cd^{\geq 0})$ be an algebraic $t$-structure on $\cd^b(\mod A)$. Then, according to Theorem~\ref{t:koenigy}, the heart is equivalent to $\mod B$ for some finite-dimensional algebra $B$. In particular, $\mod B$ embeds into $\cd^b(\mod A)$. We compare $\mod B$ with $\mod A$. The following  questions are natural.

\begin{itemize}
\item[(a)] Let $Q$ be a connected Dynkin quiver and $A=kQ$. Is it true that $B$ is either derived equivalent to $A$ or $B$ is not connected? This question was proposed by Adam-Christiaan van Roosmalen.
\item[(b)] Is $B$ less complicated than $A$ in terms of derived representation type?
\item[(c)] If $A$ has finite global dimension, does $B$ have finite global dimension?
\item[(d)] If $A$ is piecewise hereditary, is $B$ piecewise hereditary? Thanks to \cite{HappelZacharia08}, this is equivalent to: if $A$ has finite strong global dimension, does $B$ have finite strong global dimension? This question was proposed by Jorge Vit\'oria.
\end{itemize}

In \cite{BuanZhou18}, Buan and Zhou showed that the question (c) in general has a negative answer. In this section, we construct concrete examples showing that the questions (a), (b) and (d) in general have negative answers. We also give a family of examples in which $A$ and $B$ are derived equivalent, while the embedding $\mod B\hookrightarrow \cd^b(\mod A)$ induced by the given $t$-structure does not extend to a derived equivalence.

\subsection{Non-broken hearts}
We study question (a) for general connected quivers without oriented cycles. 
A necessary and sufficient condition for the heart of a bounded $t$-structure on $\cd^b(\mod kQ)$ to be derived equivalent to $kQ$ is given in \cite{StanleyvanRoosmalen16}.

\smallskip
Adam-Christiaan van Roosmalen informed us that it has a positive answer for quivers of type $\mathbb{A}$ and for $n$-Kronecker quivers. In both cases, silting objects of $\ch^b(\proj A)$ are of the form $M=\Sigma^{p(1)}T_1\oplus\ldots\oplus \Sigma^{p(n)}T_n$, where $T_1\oplus\ldots\oplus T_n$ is a tilting complex with $\Hom(T_j,T_i)=0$ for $j>i$, and $p\colon\{1,\ldots,n\}\to \mathbb{N}\cup\{0\}$ is a non-decreasing function with $p(1)=0$ (for type $\mathbb{A}$ this is \cite[Theorem 5.3]{KellerVossieck88}; for $n$-Kronecker quivers, this follows from \cite[Proposition 3.11]{AiharaIyama12} and \cite[Lemma 4]{Ringel94}). If $p$ is constant, then $M$ is a tilting complex, and $\End(M)$ is derived equivalent to $A$. If $p$ is not constant, then there exists $i_0\in\{1,\ldots,n\}$ such that $p(i_0)<p(i_0+1)$. It follows that $\End(M)$ is not connected, since $\Hom(\Sigma^{p(1)}T_1\oplus\ldots\oplus \Sigma^{p(i_0)}T_{i_0},\Sigma^{p(i_0+1)}T_{i_0+1}\oplus\ldots\oplus \Sigma^{p(n)}T_n)=0$. For $n$-Kronecker quivers the conclusion that question (a) has a positive answer can also be obtained as a consequence of the classification of bounded $t$-structures (\cite[Theorem 6.7]{GorodentsevKuleshovRudakov04} and \cite[Theorem 8.13]{StanleyvanRoosmalen19}).

\smallskip
The two examples below show that question (a) has a negative answer for quivers of type $\mathbb{D}_n$ ($n\geq 4$), of tree type $\mathbb{T}_{23(n-3)}$ and $\mathbb{T}_{2^{m+1},n-2}$ and of type $\tilde{\mathbb{A}}_n$ ($n\geq 2$). 
Notice that $\mathbb{E}_6=\mathbb{T}_{233}$, $\mathbb{E}_7=\mathbb{T}_{234}$ and $\mathbb{E}_8=\mathbb{T}_{235}$.

\begin{example}\label{e:a-n}
Let $\Lambda(n,r)$ ($n>r\geq 3$) and $\Lambda'_m(n,n-1)$ 
be as in Section~\ref{s:ADE}. Applying
Theorem~\ref{t:hereditary-heart} to $\Lambda'_m(n,n-1)$ and the
piecewise hereditary $\Lambda(n,r)$'s yields the following: 
\begin{itemize}
\item[--] Let $A$ be piecewise hereditary of type $\mathbb{D}_n$ ($n\geq 4$). Then
$\cd^b(\mod A)$ admits a bounded $t$-structure whose heart is
piecewise hereditary of type $\mathbb{A}_n$.
\item[--]
Let $A$ be piecewise hereditary of type
 $\mathbb{T}_{23(n-3)}$ ($n\geq 6$). Then
$\cd^b(\mod A)$ admits a bounded $t$-structure whose heart is
piecewise hereditary of type $\mathbb{A}_n$.
\item[--]
Let $A$ be piecewise hereditary of type
 $\mathbb{X}_{23(n-4)}$ ($n\geq 7$). Then
$\cd^b(\mod A)$ admits a bounded $t$-structure whose heart is
piecewise hereditary of type $\mathbb{A}_n$.
\item[--]
Let $A$ be piecewise hereditary of type
$\mathbb{X}_{244}$. Then
$\cd^b(\mod A)$ admits a bounded $t$-structure whose heart is
piecewise hereditary of type $\mathbb{A}_9$.
\item[--] Let $A$ be piecewise hereditary of type $\mathbb{X}_{245}$. Then
$\cd^b(\mod A)$ admits a bounded $t$-structure whose heart is
piecewise hereditary of type $\mathbb{A}_{10}$.
\item[--] Let $A$ be piecewise hereditary
of type $\mathbb{T}_{2,2,2,n-2}$. Then $\cd^b(\mod A)$ admits a
bounded $t$-structure whose heart is piecewise hereditary of type
$\mathbb{D}_n$. This follows from Proposition~\ref{p:extend-ADE}. More generally, let $A$ be piecewise hereditary of type $\mathbb{T}_{2^{m+1},n-2}$. Then $\cd^b(\mod A)$ admits a bounded $t$-structure whose heart is piecewise hereditary of type $\mathbb{T}_{2^m,n-1}$.
\end{itemize}
\end{example}

\begin{example}\label{e:a-n-2} Let $p,q\in\mathbb{N}$.
Let $A$ be the path algebra of the quiver $Q$
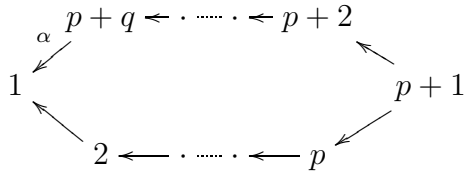
\[\xymatrix@R=0.7pc@C=0.7pc{&p+q\ar[dl]_{\alpha}&\cdot\ar[l]\ar@{.}[r]&\cdot&p+2\ar[l]&\\
1&&&&&p+1\ar[dl]\ar[ul].\\
&2\ar[ul]&\cdot\ar[l]\ar@{.}[r]&\cdot&p\ar[l]&}\]
Let $P_1,\ldots,P_{p+q}$ be the indecomposable projective modules corresponding to the vertices $1,\ldots,p+q$.
For $i=1,\ldots,p+q$, let
\[M_i=\begin{cases}
\mathrm{Cone}(P_{p+q}\stackrel{\alpha}{\longrightarrow}P_1) & \text{ if } i=1\\
P_i & \text{ if } 2\leq i\leq p+q-1\\
\Sigma P_{p+q} & \text{ if } i=p+q.
\end{cases}\]
Then $M=\bigoplus_{i=1}^{p+q}M_i$ is a silting object in
$\ch^b(\proj A)$. Indeed, direct computation shows that the
graded endomorphism algebra
$\bigoplus_{m\in\mathbb{Z}}\Hom(M,\Sigma^m M)$ is isomorphic to the
quotient of the path algebra of the graded quiver $Q'$
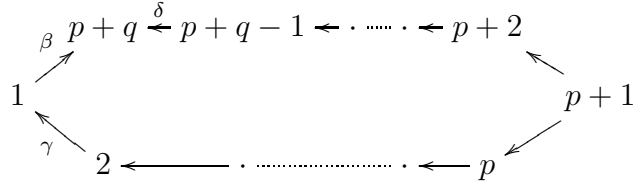
\[\xymatrix@R=0.7pc@C=0.7pc{&p+q&p+q-1\ar[l]_(0.6){\delta}&\cdot\ar[l]\ar@{.}[r]&\cdot&p+2\ar[l]&\\
1\ar[ru]^{\beta}&&&&&&p+1\ar[dl]\ar[ul].\\
&2\ar[ul]^{\gamma}&\cdot\ar[l]\ar@{.}[rr]&&\cdot&p\ar[l]&}\] modulo
$\beta\gamma$, where $\delta$ is of degree $-1$ and all other arrows
are of degree $0$. Note that $Q'$ differs from $Q$ by two arrows:
the arrow $\alpha$ from $p+q$ to $1$ is reversed and the degree of
the arrow from $p+q-1$ to $p+q$ decreases by $1$. The
endomorphism algebra $B=\End(M)$, which is the degree $0$ component of
the graded endomorphism algebra, is the quotient of the path algebra of the quiver
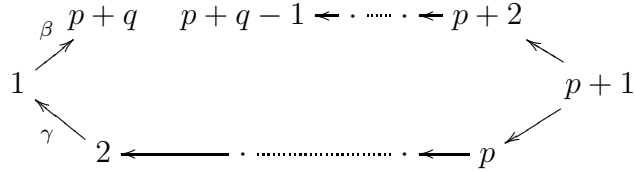
\[\xymatrix@R=0.7pc@C=0.7pc{&p+q&p+q-1&\cdot\ar[l]\ar@{.}[r]&\cdot&p+2\ar[l]&\\
1\ar[ru]^{\beta}&&&&&&p+1\ar[dl]\ar[ul].\\
&2\ar[ul]^{\gamma}&\cdot\ar[l]\ar@{.}[rr]&&\cdot&p\ar[l]&}\]
(obtained from $Q'$ by deleting $\delta$) modulo $\beta\gamma$. By
Theorem~\ref{t:koenigy}, there is a bounded $t$-structure whose heart is naturally equivalent to $\mod B$. Notice that $B$ is
connected, and is derived equivalent to the path algebra of any
quiver of type $\mathbb{A}_{p+q}$.
\end{example}

\subsection{Derived representation type}
It is well-known that quivers of different type are not derived equivalent. In this subsection we will see that quivers of different type sometimes are related via $t$-structures. In some sense an affine quiver can be found inside the derived category of a Dynkin quiver and a wild quiver can be found inside the derived category of an affine quiver. In particular, the general answer to question (b) is negative: $B$ can be derived representation-infinite even if $A$ is derived representation-finite; $B$ can be derived wild even if $A$ is derived tame.

\begin{example}\label{e:e-6}
Let $Q'$ be a quiver of type $\mathbb{E}_6$. Let $A_1$ be the quotient of the path algebra of the quiver
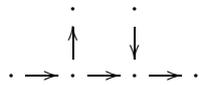
\[
\xymatrix@R=1pc@C=1pc{
& \cdot &\cdot\ar[d]& \\
\cdot\ar[r] & \cdot\ar[u]\ar[r] & \cdot\ar[r] & \cdot} 
\] 
modulo the unique path of length $3$ and let $A_2$ be the quotient of the path algebra of the quiver
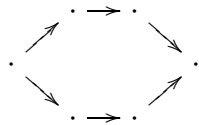
\[
\xymatrix@R=0.6pc@C=1pc{&\cdot\ar[r]&\cdot\ar[rd]&\\
\cdot\ar[ru]\ar[rd]&&&\cdot\\
&\cdot\ar[r]&\cdot\ar[ru]&}
\]
modulo the sum of the two paths of length $3$. They are
tilted algebras of type $\mathbb{E}_6$, see~\cite[Section
4.2]{HappelRingel80}. So both $A_1$ and $A_2$ are
derived equivalent to $kQ'$.

Note that the quiver for $A_1$ is affine of type $\tilde{
\mathbb{D}}_5$ and that for $A_2$ is affine of type $\tilde{
\mathbb{A}}_{3,3}$. Then by Theorem~\ref{t:hereditary-heart} there
are two bounded $t$-structures on $\cd^b(\mod kQ')$ whose hearts are
derived equivalent to a quiver of type $\tilde{ \mathbb{D}}_5$ and a
quiver of type $\tilde {\mathbb{A}}_{3,3}$, respectively.
\end{example}

\begin{example}\label{e:happel-vossieck-list}
Going through the Happel--Vossieck list~\cite{HappelVossieck83}, we
obtain examples similar to those in Example~\ref{e:e-6}. This time
we obtain wild quivers from affine ones. For example, let $Q$ be the
wild quiver
\[\xymatrix@R=1pc@C=1.5pc{\cdot\ar[r]&\cdot\ar[r]&\cdot\ar[d]^\beta&\cdot\ar[l]_\alpha\\
\cdot\ar[r]&\cdot\ar[r]&\cdot\ar[r]^\gamma&\cdot}\] and let
$A$ be the quotient of $kQ$ by the ideal generated by
$\gamma\beta\alpha$. Then $A$ is a tilted algebra of type
$\tilde{\mathbb{E}}_7$. So, in the bounded derived category of a
quiver of type $\tilde{\mathbb{E}}_7$ there is a bounded
$t$-structure whose heart is derived equivalent to $kQ$.
\end{example}

\subsection{Global dimension}
If $A$ is hereditary, then $B$ necessarily has finite global dimension. Indeed, in this case, the indecomposable direct summands of any basic silting object form an exceptional sequence, by \cite[Proposition 3.11]{AiharaIyama12}. As a consequence, the quiver of $B$ (as the endomorphism algebra of a silting object) is directed and hence $B$ has finite global dimension. However, the general answer to question (c) is negative: Buan and Zhou give examples in \cite[Section 5.3]{BuanZhou18} to show that once the global dimension of $A$ increases to $3$, $B$ may have infinite global dimension. 

\subsection{Strong global dimension}

This example shows that the heart of a bounded $t$-structure on the bounded derived category of a piecewise hereditary algebra is not necessarily piecewise hereditary.

\begin{example}\label{ex:strong-global-dimension}

Let $Q$ be the quiver
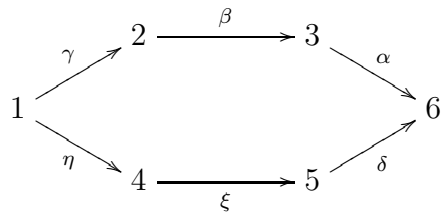
\[\xymatrix@C=1pc@R=1pc{&&2\ar[rrr]^{\beta}&&&3\ar[rrd]^{\alpha}\\
1\ar[rrd]_\eta\ar[rru]^\gamma&&&&&&&6\\
&&4\ar[rrr]_\xi&&&5\ar[rru]_\delta}\]
and let $A$ be the quotient of $kQ$ modulo the ideal generated by $\beta\gamma$ and $\delta\xi\eta$. For a vertex $i$, let $S_i$ denote the simple $A$-module supported at $i$ and $P_i$ denote its projective cover. Let   $T=\bigoplus_{i=1}^6 T_i$, where
\begin{eqnarray*}
T_1 &=& S_4,\\
T_2 &=& \ker(P_3\stackrel{\alpha}{\longrightarrow}P_6),\\
T_3 &=& \cok(P_1\stackrel{{\gamma\choose\xi\eta}}{\longrightarrow}P_2\oplus P_5),\\
T_4 &=& S_2,\\
T_5 &=& P_3,\\
T_6 &=& P_6.
\end{eqnarray*}
Direct calculation shows that $T$ is a tilting $A$-module and its endomorphism algebra is the path algebra $A'$ of the quiver $Q'$
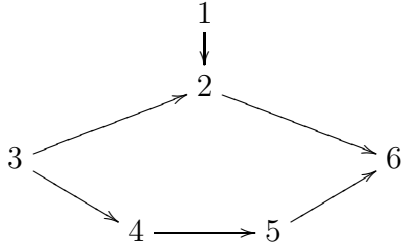
\[\xymatrix@C=1pc@R=1pc{&&&1\ar[d]&&&\\
&&&2\ar[rrrd]&&&\\
3\ar[rrru]\ar[rrd]&&&&&&6\\
&&4\ar[rr]&&5\ar[rru]&&}\]
In particular, $A'$ is piecewise hereditary.

Applying Theorem~\ref{t:heart-2} to $A$, we obtain a bounded $t$-structure on $\cd^b(\mod A)$ whose heart is equivalent to $\mod A^!$. By definition $A^!$ is the quotient of the path algebra of the quiver
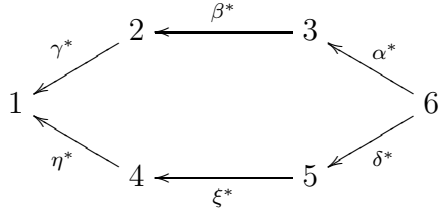
\[\xymatrix@C=1pc@R=1pc{&&2\ar[lld]_{\gamma^*}&&&3\ar[lll]_{\beta^*}\\
1&&&&&&&6\ar[llu]_{\alpha^*}\ar[lld]^{\delta^*}\\
&&4\ar[llu]^{\eta^*}&&&5\ar[lll]^{\xi^*}}\]
modulo all paths of length $2$ other than $\gamma^*\beta^*$. Therefore $A^!$ is a gentle one-cycle algebra, and the Auslander--Reiten quiver of $\cd^b(\mod A^!)$ has three connected components, two of which have shape $\mathbb{Z}A_\infty$ and one of which has shape $\mathbb{Z}A_\infty^\infty$, by~\cite[Section 7]{Avella-AlaminosGeiss08} and~\cite[Theorem B]{BobinskiGeissSkowronski04}. It follows that $A^!$ is not piecewise hereditary since the Auslander--Reiten quiver of the bounded derived category of a hereditary category either is connected or has infinitely many connected components.

\smallskip
We remark that the derived equivalence between $A$ and $A'$ can also be obtained by quiver mutation, see Appendix~\ref{a:quiver-mutation}.
\end{example}

\subsection{Further examples}
Martin Kalck suggests the following example.

\begin{example}\label{e:canonical-alg}
Let $p\geq 3$, $r\geq 2$ and let $\mathbb{X}=\mathbb{X}_{p^r}$ be a weighted
projective line of type $(p^r)$. It follows
from~\cite[Section 4]{GeigleLenzing87} that $\mathbb{X}$ is derived
equivalent to a canonical algebra $\Lambda$ of type $(p,\ldots,p)$
($r$ copies), whose quiver is
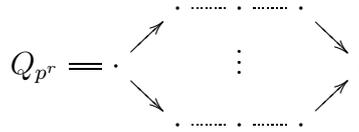
\[\xymatrix@R=0.2pc@C=1pc{&&\cdot\ar@{.}[r]&\cdot\ar@{.}[r]&\cdot\ar[dr]&\\
Q_{p^r}\ar@{=}[r]&\cdot\ar[ur]\ar[dr]&&\vdots&&\cdot\\
&&\cdot\ar@{.}[r]&\cdot\ar@{.}[r]&\cdot\ar[ur]&}\] (there are $r$
arms all of which consist of $p$ arrows). It satisfies a relation that is a
linear combination of the longest paths. By
Theorem~\ref{t:hereditary-heart}, $\cd^b(\coh\mathbb{X})$ admits a
bounded $t$-structure whose heart is derived equivalent to
$kQ_{p^r}$. This shows that $\coh\mathbb{X}$ and $\mod kQ_{p^r}$ are
closely related, which may suggest some interesting connection between
their Hall algebras.
\end{example}

In 2014, Martin Kalck observed the following surprising phenomenon\footnote{We are grateful to Martin Kalck for the permission to include this example.}: Let $A$ be the quotient of the path algebra of the quiver
\[
\xymatrix{
1 \ar[r]^\alpha& 2 \ar@<.7ex>[r]^\gamma & 3\ar@<.7ex>[l]^\beta
}
\]
modulo the two paths $\gamma\beta$ and $\beta\gamma\alpha$. Let $M=P_1\oplus P_2\oplus \Cone(P_3\stackrel{\beta}{\to}P_2)$, where $P_1,P_2,P_3$ are the indecomposable projetcive $A$-modules corresponding to the vertices $1,2,3$, respectively. Then $M$ is the left mutation of $A_A$ at $P_3$ and the space $\Hom(P_1,\Sigma^{-1}\Cone(P_3\stackrel{\beta}{\to}P_2))$ is 1-dimensional, and therefore $M$ is silting but not tilting, however, $\End(M)$ is isomorphic to $A$. Recently this phenomenon is found to occur already in Dynkin type $\mathbb{D}_5$, as a consequence of the classification of silted algebras of a hereditary algebra of type $\mathbb{D}_5$ in \cite{Xing20}. When trying to give a proof using Theorem~\ref{t:hereditary-heart} we find that this phenomenon actually occurs in all Dynkin types $\mathbb{D}_n~(n\geq 5)$ and $\mathbb{E}_{6,7,8}$. We believe that the derived equivalence between $A$ and $A'$ in the next two examples are well-known, nevertheless, we give a proof.

\begin{example}\label{ex:auto-silting-type-D}
Let $n\geq 5$, $Q$ be the following quiver of type $\mathbb{D}_n$
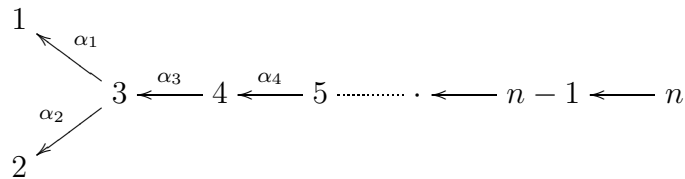
\[
\xymatrix@R=1pc{
1\\
& 3\ar[lu]_{\alpha_1}\ar[ld]_{\alpha_2} & 4\ar[l]_{\alpha_3} & 5\ar[l]_{\alpha_4} & \cdot\ar@{.}[l] & n-1\ar[l] & n\ar[l]\\
2
}
\]
and let $A=kQ$ be the path algebra of $Q$. \emph{Then there exists a silting object in $\ch^b(\proj A)$ which is not tilting, and its endomorphism algebra is derived equivalent to $A$. The heart $\cb$ of the corresponding $t$-structure is derived equivalent to $\mod A$ but the embedding $\cb\hookrightarrow\cd^b(\mod A)$ does not extend to a derived equivalence\footnote{\cite[Example 4.10]{StanleyvanRoosmalen16} is an example of the same kind for infinite-dimensional algebras}.}

Let $M=\bigoplus_{i=1}^n M_i$ be the $A$-module with 
\begin{align*}
M_1&=\cok(P_4\stackrel{\alpha_1\alpha_3}{\longrightarrow}P_2),\\
M_2&=\cok(P_3\stackrel{\alpha_1\choose\alpha_2}{\longrightarrow}P_1\oplus P_2),\\
M_3&=\cok(P_4\stackrel{\alpha_1\alpha_3\choose\alpha_2\alpha_3}{\longrightarrow}P_1\oplus P_2),\\
M_4&=P_2,\\
M_i&=P_i \text{ for any }i\geq 5.
\end{align*}
Direct calculation shows that $M$ is a tilting module and its endomorphism algebra is isomorphic  to $A'=kQ/(\alpha_1\alpha_3\alpha_4)$. 
Now the assertion is obtained by applying Theorem~\ref{t:hereditary-heart} and its proof to $A'$ in conjunction with the derived equivalence between $A$ and $A'$.
\end{example}

\begin{example}\label{ex:auto-silting-type-E}
Let $n\geq 6$, $Q$ be the quiver
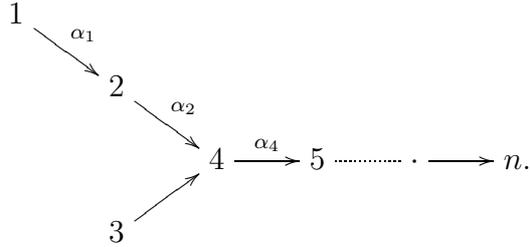
\[
\xymatrix@R=1pc{
1\ar[rd]^{\alpha_1} & \\
& 2 \ar[rd]^{\alpha_2} & \\
& & 4\ar[r]^{\alpha_4} & 5\ar@{.}[r] & \cdot\ar[r] & n. \\
& 3\ar[ru] & 
}
\]
and let $A=kQ$. \emph{Then there exists a silting object in $\ch^b(\proj A)$ which is not tilting, and its endomorphism algebra is derived equivalent to $A$. The heart $\cb$ of the corresponding $t$-structure is derived equivalent to $\mod A$ but the embedding $\cb\hookrightarrow\cd^b(\mod A)$ does not extend to a derived equivalence.}

Let $A'=kQ/(\alpha_4\alpha_2\alpha_1)$. We claim that $A$ and $A'$ are derived equivalent. Instead of constructing a tilting module as in Example~\ref{ex:auto-silting-type-D}, we will prove this in the end of Appendix~\ref{a:quiver-mutation} by using quiver mutation. Now the assertion is obtained by applying Theorem~\ref{t:hereditary-heart} and its proof to $A'$ in conjunction with the derived equivalence between $A$ and $A'$.
\end{example}

\appendix

\section{An alternative proof of Proposition~\ref{p:extend-ADE}}
\label{a:quiver-mutation}

The derived equivalence in Proposition~\ref{p:extend-ADE} for $m=2$ and that in Example~\ref{ex:strong-global-dimension} were first established using quiver mutation.
In this appendix we describe this alternative approach, which provides quite some fun and shows the power of quiver mutation
in the representation theory of finite-dimensional algebras\footnote{In particular, using Keller's quiver mutation applet \cite{KellerQuiverMutation}, one can show the piecewise heredity of some finite-dimensional algebras by just clicking on the mouse.}.  We use
the following result of Mizuno~\cite{Mizuno14}. Terminologies will
be explained after stating the result.

\begin{theorem}\emph{(\cite[Theorem 2.1]{Mizuno14})}\label{t:mizuno}
Let $\Lambda=\widehat{kQ}/\overline{(R)}$ be a finite-dimensional algebra,
where $R=\{r_1,\ldots,r_m\}$ is a set of minimal relations for $\Lambda$. Define a graded quiver
with potential $(Q^a,W^a)$ by
\begin{itemize}
 \item[$\cdot$] $Q^a_0=Q_0$,
 \item[$\cdot$] $Q^a_1=Q_1\cup\{\rho_i\colon t(r_i)\rightarrow s(r_i)\mid i=1,\ldots,m\}$ with $\deg(\alpha)=0$ for $\alpha\in Q_1$
and $\deg(\rho_i)=1$ for $i=1,\ldots,m$.
 \item[$\cdot$] $W^a=\sum_{i=1}^m \rho_i r_i$.
\end{itemize}
Let $i$ be a vertex of $Q$ such that the corresponding simple
module is projective and has injective dimension $1$ or $2$. Let $T$
be the corresponding APR tilting module. Then the endomorphism
algebra of $T$ is isomorphic to the degree $0$ component of the
graded algebra $\widehat{J}(\mu^L_i(Q^a,W^a))$, the complete Jacobian algebra of the left
mutation of $(Q^a,W^a)$ at $i$.
\end{theorem}

A \emph{graded quiver with potential} is a pair $(Q,W)$, where $Q$
is a finite graded quiver and $W$ is a homogeneous linear
combination of cycles of $Q$. For our purpose we assume that $W$ is
homogeneous of degree $1$. Let $\widehat{kQ}$ be the complete path algebra of $Q$, \ie the
completion of the path algebra $kQ$ with respect to the ideal generated by all arrows
in the category of graded algebras. The algebra $\widehat{kQ}$ is naturally equipped with a topology. For any
subset $S$ of $\widehat{kQ}$, we denote by $\overline{S}$ the closure of $S$. For an arrow $\alpha$ of $Q$, define
$\del_\alpha\colon\widehat{kQ}\rightarrow \widehat{kQ}$ as the unique continuous linear map which takes
a path $c$ to the sum $ \sum_{c=u p v} vu $ taken over all
decompositions of the path $c$ (where $u$ and $v$ are possibly trivial
paths). The \emph{complete Jacobian algebra} $J(Q,W)$ is the graded algebra
\[\widehat{J}(Q,W)=\widehat{kQ}/\overline{(\del_\alpha W,\alpha\in Q_1)}.\]
The potential $W$ is said to be \emph{rigid} if
$\widehat{J}(Q,W)/\overline{[\widehat{J}(Q,W),\widehat{J}(Q,W)]}=\prod_{Q_0}K$.

Let $Q$ be a graded quiver without loops or 2-cycles. Let $i$ be a
vertex of $Q$. We define the \emph{left mutation} ${\mu}^L_{i}(Q)$
of $Q$ at $i$ as the graded quiver obtained from $Q$ in the
following three steps~\cite{FominZelevinsky02,AmiotOppermann14}
\begin{itemize}
\item[--] For each arrow $\beta$ with target $i$ and each
arrow $\alpha$ with source $i$, add a new arrow $[\alpha\beta]$ of degree $\deg(\alpha)+\deg(\beta)$ from the source of
$\beta$ to the target of $\alpha$.
\item[--] Replace each arrow $\alpha$ with source $i$ with
an arrow $\alpha^\star$ in the opposite direction of degree $-\deg(\alpha)$;
replace each arrow $\beta$ with target $i$ with
an arrow $\beta^\star$ in the opposite direction of degree $1-\deg(\beta)$.
\item[--] Remove all arrows in a maximal set of pairwise disjoint 2-cycles.
\end{itemize}

This can be extended to a left mutation $\mu^L_i(Q,W)$ of graded quivers
with potential, see~\cite{DerksenWeymanZelevinsky08,AmiotOppermann14}. We will not need the precise definition. The following result of Derksen, Weyman and Zelevinsky
is sufficient for us.

\begin{proposition}\label{p:rigid-qp} \emph{(\cite[Corollary 8.2]{DerksenWeymanZelevinsky08})}
Assume that $(Q,W)$ is a rigid graded quiver with potential such
that $W$ contains no 2-cycles. Then $Q$ has no 2-cycles and the
mutated graded quiver with potential $\mu^L_i(Q,W)$ is again rigid.
In particular, the graded quiver of $\mu^L_i(Q,W)$ coincides with
the mutated graded quiver $\mu^L_i(Q)$.
\end{proposition}

We have the following useful corollary.

\begin{corollary}\label{cor:iterated-APR-tilting-via-mutation}
Let $\Lambda$ and $(Q^a,W^a)$ be as in Theorem~\ref{t:mizuno}. Assume further that $\Lambda$ has global dimension at most two and that $(Q^a,W^a)$ is a rigid quiver with potential and $W^a$ contains no 2-cycles. Let $i_1,\ldots,i_t$ be a sequence of vertices of $Q^a$ such that $i_s$ is a source of the degree $0$ part of $\mu_{i_{s-1}}\cdots\mu_{i_1}(Q^a)$ for any $1\leq s\leq t$. If $Q'=\mu_{i_t}\cdots\mu_{i_1}(Q^a)$ is acyclic, then $\Lambda$ and $kQ'$ are related by a sequence of APR tilts, in particular, they are derived equivalent.
\end{corollary}
\begin{proof}
Under the assumptions the simple $\Lambda$-module corresponding to $i_1$ is projective and has injective dimension at most $2$. Let $\Lambda_1$ be the degree $0$ component of $\widehat{J}(\mu_{i_1}^L(Q^a,W^a))$. By Theorem~\ref{t:mizuno},  the algebras $\Lambda$ and $\Lambda_1$ are related by an APR tilt. Moreover, by Proposition~\ref{p:rigid-qp}, the quiver with potential $\mu_{i_1}^L(Q^a,W^a)$ is again rigid, its quiver is $\mu_{i_1}^L(Q^a)$ and its potential contains no 2-cycles. By~\cite[Proposition 1.15]{AuslanderPlatzeckReiten79}, $\Lambda_1$ has global dimension at most two. If $t=1$, then $\mu_{i_1}^L(Q^a,W^a)=(Q',0)$ and therefore $\widehat{J}(Q',0)=kQ'$, because $Q'=\mu_{i_1}^L(Q^a)$ is acyclic. If $t\geq 2$, then by induction $\Lambda_1$ and $kQ'$ are related by a sequence of APR tilts, and the proof finishes. 
\end{proof}

In the rest of this appendix, we apply Corollary~\ref{cor:iterated-APR-tilting-via-mutation} to show the desired derived equivalence in Proposition~\ref{p:extend-ADE} and Examples~\ref{ex:strong-global-dimension}~and~\ref{ex:auto-silting-type-E}. As we will always perform APR tilting, it follows by induction that all arrows appearing below are either in degree $0$, represented as solid arrows, or in degree $1$, represented as dashed arrows.

\begin{proof}[Proof of Proposition~\ref{p:extend-ADE} (case $m=2$)]
The method works for general $m$ as well, but the quivers below will become messy. So we assume $m=2$.

Let $\Lambda=\Lambda'_2(n,n-1)$. According to Theorem~\ref{t:mizuno}, the associated graded quiver $Q^a$ is 
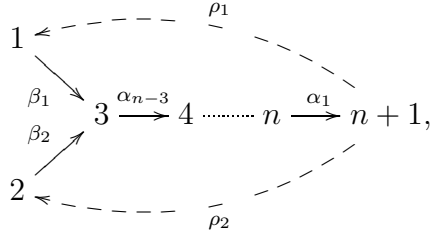
\[\xymatrix@R=1pc@C=1.5pc{1\ar[dr]_{\beta_1}\\
&3\ar[r]^{\alpha_{n-3}}&4\ar@{.}[r]&n\ar[r]^(0.4){\alpha_1}&n+1\ar@/^20pt/@{-->}[lllld]^{\rho_2}\ar@{-->}@/_20pt/[llllu]_{\rho_1},\\
2\ar[ur]^{\beta_2} }\] where $\rho_1$ and $\rho_2$ are in degree $1$
and all other arrows are in degree $0$, and the associated potential is
$$W^a=\rho_1\alpha_1\cdots\alpha_{n-3}\beta_1+\rho_2\alpha_1\cdots\alpha_{n-3}\beta_2.$$
This quiver with potential is rigid, because the two minimal oriented cycles both vanish in the complete Jacobian algebra. Thus it follows by Corollary~\ref{cor:iterated-APR-tilting-via-mutation} and the following claim that $\Lambda$ is
piecewise hereditary of type $\mathbb{T}_{2,2,2,n-3}$, as desired. Note that each left mutation is taken at a vertex which becomes a source after removing the dashed arrows.

Claim: For $3\leq m\leq n-1$ let $\nu_m^L=\mu^L_1\cdots\mu^L_m$. The
iterated left mutation
$$\mu_3^L\cdots\mu_n^L\nu_{n-1}^L\cdots\nu_3^L\mu_2^L\mu_1^L(Q)$$ is
the quiver
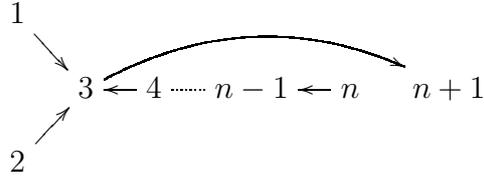
\[\xymatrix@R=1pc@C=1pc{1\ar[dr]\\
&3\ar@/^20pt/[rrrr]&4\ar[l]\ar@{.}[r]&n-1&n\ar[l]&n+1\\
2\ar[ur] }\]

We approach the claim in four steps. For concrete $n$, one can perform these mutations using Keller's quiver mutation applet \cite{KellerQuiverMutation}.

Step 1: We perform $\nu^L_3\mu^L_2\mu^L_1$.
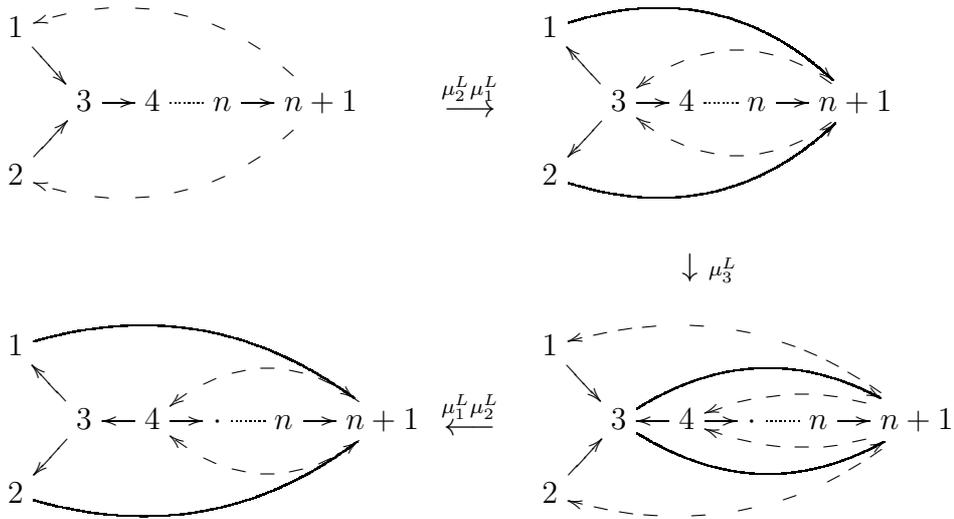
\[
\begin{picture}(400,205)
 \put(0,190){
$\xymatrix@R=1pc@C=1pc{1\ar[dr]\\
&3\ar[r]&4\ar@{.}[r]&n\ar[r]&n+1\ar@/^20pt/@{-->}[lllld]\ar@{-->}@/_20pt/[llllu]\\
2\ar[ur] }$} \put(200,190){
$\xymatrix@R=1pc@C=1pc{1\ar@/^20pt/[rrrrd]\\
&3\ar[r]\ar[lu]\ar[ld]&4\ar@{.}[r]&n\ar[r]&n+1\ar@{-->}@/_20pt/[lll]\ar@/^20pt/@{-->}[lll]\\
2\ar@/_20pt/[rrrru]
}$}
 \put(200,70){
$\xymatrix@R=1pc@C=1pc{1\ar[dr]\\
&3\ar@/_20pt/[rrrr]\ar@/^20pt/[rrrr]&4\ar[l]\ar[r]&\cdot\ar@{.}[r]&n\ar[r]&n+1\ar@/^20pt/@{-->}[llllld]\ar@{-->}@/_20pt/[lllllu]\ar@{-->}@/_10pt/[lll]\ar@/^10pt/@{-->}[lll]\\
2\ar[ur] }$} \put(0,70){
$\xymatrix@R=1pc@C=1pc{1\ar@/^20pt/[rrrrrd]\\
&3\ar[lu]\ar[ld]&4\ar[l]\ar[r]&\cdot\ar@{.}[r]&n\ar[r]&n+1\ar@{-->}@/_20pt/[lll]\ar@/^20pt/@{-->}[lll]\\
2\ar@/_20pt/[rrrrru] }$}
\put(170,160){$\stackrel{\scriptstyle{\mu^L_2\mu^L_1}}{\longrightarrow}$}
\put(260,100){${\downarrow}$}
\put(270,100){${\scriptstyle{\mu^L_3}}$}
\put(170,40){$\stackrel{\scriptstyle{\mu^L_1\mu^L_2}}{\longleftarrow}$}
\end{picture}\]

Step 2: We claim that for $4\leq m\leq n-1$ the composite mutation $\nu^L_m=\mu^L_1\cdots\mu^L_m$ takes
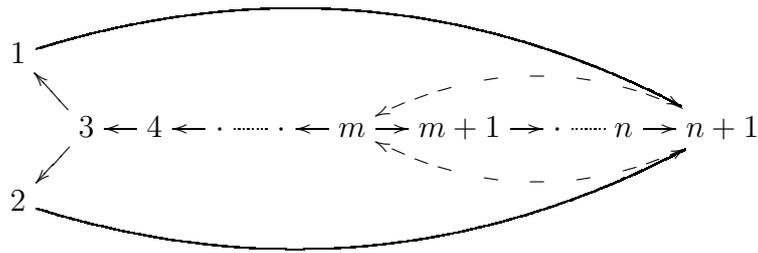
\[\xymatrix@R=1pc@C=1pc{\\ 1\ar@/^30pt/[rrrrrrrrrd]\\
&3\ar[lu]\ar[ld]&4\ar[l]&\cdot\ar[l]\ar@{.}[r]&\cdot&m\ar[l]\ar[r]&m+1\ar[r]&\cdot\ar@{.}[r]&n\ar[r]&n+1\ar@{-->}@/_20pt/[llll]\ar@/^20pt/@{-->}[llll]\\
2\ar@/_30pt/[rrrrrrrrru]
}\]
to
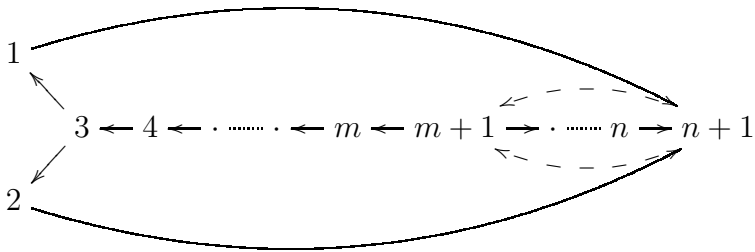
\[\xymatrix@R=1pc@C=1pc{\\ 1\ar@/^30pt/[rrrrrrrrrd]\\
&3\ar[lu]\ar[ld]&4\ar[l]&\cdot\ar[l]\ar@{.}[r]&\cdot&m\ar[l]&m+1\ar[l]\ar[r]&\cdot\ar@{.}[r]&n\ar[r]&n+1\ar@{-->}@/_15pt/[lll]\ar@/^15pt/@{-->}[lll]\\
2\ar@/_30pt/[rrrrrrrrru]\\{}
}.\]
It follows by induction on $m$ that $\nu^L_{n-1}\cdots\nu_3^L\mu_2^L\mu_1^L(Q)$ is the quiver
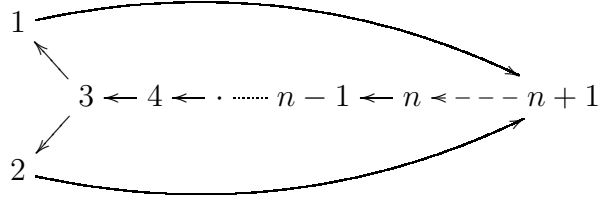
\[\xymatrix@R=1pc@C=1pc{ 1\ar@/^20pt/[rrrrrrrd]\\
&3\ar[lu]\ar[ld]&4\ar[l]&\cdot\ar[l]\ar@{.}[r]&n-1&n\ar[l]&&n+1\ar@{-->}[ll]\\
2\ar@/_20pt/[rrrrrrru] }.\]

Step 3: We prove the claim in Step 2.
Observe that for $4\leq l\leq m-1$ the mutation $\mu^L_l$ takes
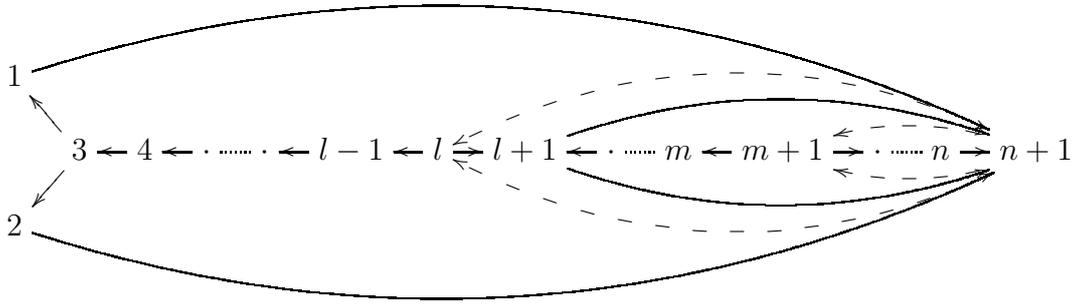
\[\xymatrix@R=1pc@C=0.9pc{\\ 1\ar@/^40pt/[rrrrrrrrrrrrrd]\\
&3\ar[lu]\ar[ld]&4\ar[l]&\cdot\ar[l]\ar@{.}[r]&\cdot&l-1\ar[l]&l\ar[l]\ar[r]
&l+1\ar@/^20pt/[rrrrrr]\ar@/_20pt/[rrrrrr] &\cdot\ar[l]\ar@{.}[r]&m&m+1\ar[l]\ar[r]&\cdot\ar@{.}[r]&n\ar[r]&n+1\ar@{-->}@/_30pt/[lllllll]\ar@/^30pt/@{-->}[lllllll]\ar@{-->}@/_10pt/[lll]\ar@/^10pt/@{-->}[lll]\\
2\ar@/_40pt/[rrrrrrrrrrrrru]\\{}
}\]
to
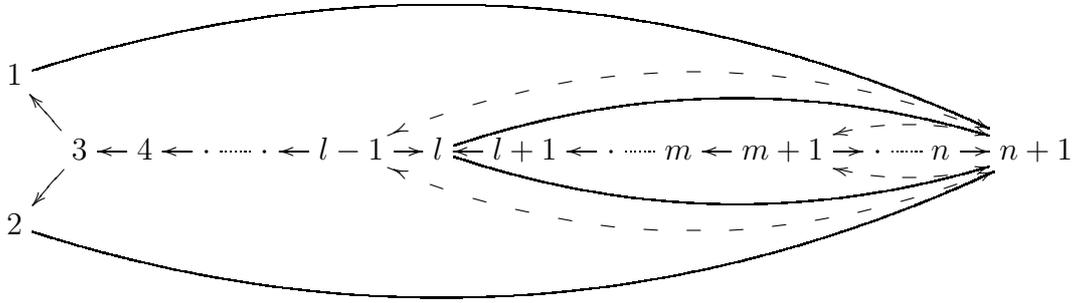
\[\xymatrix@R=1pc@C=0.9pc{\\ 1\ar@/^40pt/[rrrrrrrrrrrrrd]\\
&3\ar[lu]\ar[ld]&4\ar[l]&\cdot\ar[l]\ar@{.}[r]&\cdot&l-1\ar[l]\ar[r]&l\ar@/^20pt/[rrrrrrr]\ar@/_20pt/[rrrrrrr]
&l+1\ar[l] &\cdot\ar[l]\ar@{.}[r]&m&m+1\ar[l]\ar[r]&\cdot\ar@{.}[r]&n\ar[r]&n+1\ar@{-->}@/_30pt/[llllllll]\ar@/^30pt/@{-->}[llllllll]\ar@{-->}@/_10pt/[lll]\ar@/^10pt/@{-->}[lll]\\
2\ar@/_40pt/[rrrrrrrrrrrrru]\\{}
}\]
Therefore
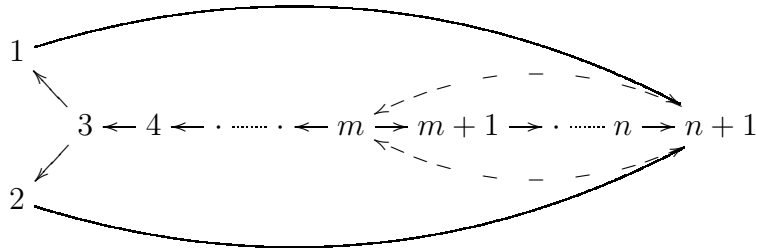
\[\xymatrix@R=1pc@C=1pc{\\ 1\ar@/^30pt/[rrrrrrrrrd]\\
&3\ar[lu]\ar[ld]&4\ar[l]&\cdot\ar[l]\ar@{.}[r]&\cdot&m\ar[l]\ar[r]&m+1\ar[r]&\cdot\ar@{.}[r]&n\ar[r]&n+1\ar@{-->}@/_20pt/[llll]\ar@/^20pt/@{-->}[llll]\\
2\ar@/_30pt/[rrrrrrrrru]\\{}
}\]
\[\downarrow \mu^L_m\]
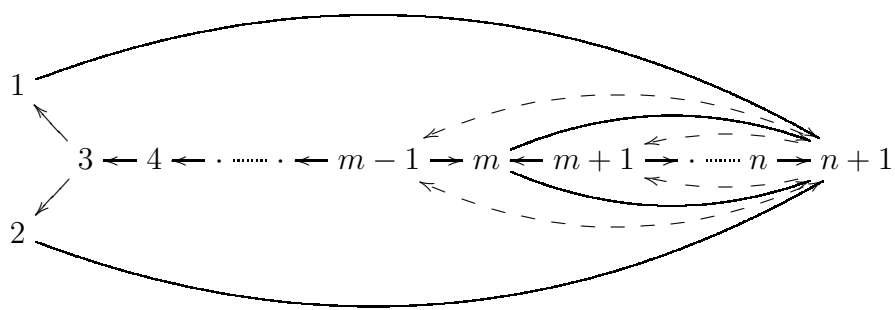
\[\xymatrix@R=1pc@C=1pc{\\ 1\ar@/^40pt/[rrrrrrrrrrd]\\
&3\ar[lu]\ar[ld]&4\ar[l]&\cdot\ar[l]\ar@{.}[r]&\cdot&m-1\ar[l]\ar[r]&m\ar@/^17pt/[rrrr]\ar@/_17pt/[rrrr]
&m+1\ar[l]\ar[r]&\cdot\ar@{.}[r]&n\ar[r]&n+1\ar@{-->}@/_25pt/[lllll]\ar@/^25pt/@{-->}[lllll]\ar@{-->}@/_10pt/[lll]\ar@/^10pt/@{-->}[lll]\\
2\ar@/_40pt/[rrrrrrrrrru]\\{}\\{}
}\]
\[\downarrow\mu^L_4\cdots\mu^L_{m-1}\]
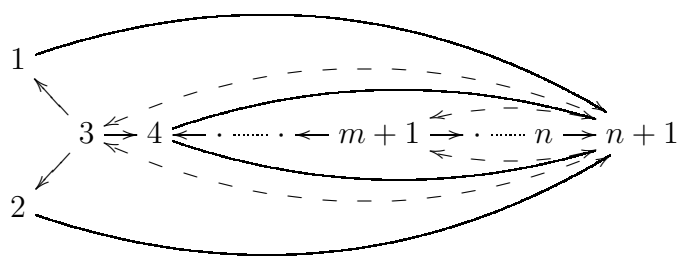
\[\xymatrix@R=1pc@C=1pc{\\ 1\ar@/^30pt/[rrrrrrrrd]\\
&3\ar[lu]\ar[ld]\ar[r]&4\ar@/^17pt/[rrrrrr]\ar@/_17pt/[rrrrrr] &\cdot\ar[l]\ar@{.}[r]&\cdot\ar@{.}[l]&m+1\ar[l]\ar[r]&\cdot\ar@{.}[r]&n\ar[r]&n+1\ar@{-->}@/_10pt/[lll]\ar@/^10pt/@{-->}[lll]\ar@{-->}@/_25pt/[lllllll]\ar@/^25pt/@{-->}[lllllll]\\
2\ar@/_30pt/[rrrrrrrru]\\{}
}\]
\[\downarrow\mu^L_3\]
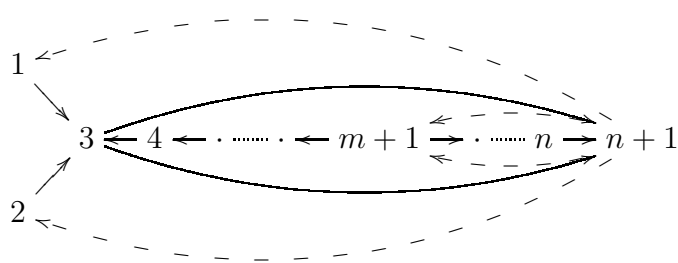
\[\xymatrix@R=1pc@C=1pc{\\ 1\ar[dr]\\
&3\ar@/^20pt/[rrrrrrr]\ar@/_20pt/[rrrrrrr]&4\ar[l] &\cdot\ar[l]\ar@{.}[r]&\cdot\ar@{.}[l]&m+1\ar[l]\ar[r]&\cdot\ar@{.}[r]&n\ar[r]&n+1\ar@{-->}@/_10pt/[lll]\ar@/^10pt/@{-->}[lll]\ar@{-->}@/_30pt/[llllllllu]\ar@/^30pt/@{-->}[lllllllld]\\
2\ar[ur]\\{}
}\]
\[\downarrow\mu^L_1\mu^L_2\]
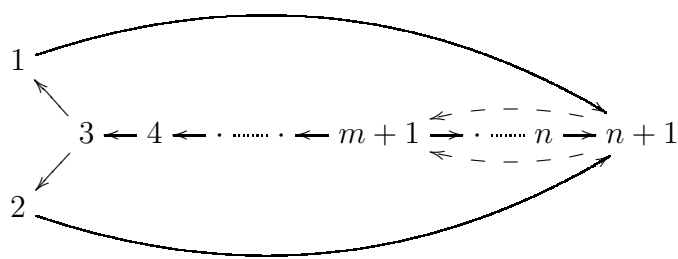
\[\xymatrix@R=1pc@C=1pc{\\ 1\ar@/^30pt/[rrrrrrrrd]\\
&3\ar[lu]\ar[ld]&4\ar[l] &\cdot\ar[l]\ar@{.}[r]&\cdot\ar@{.}[l]&m+1\ar[l]\ar[r]&\cdot\ar@{.}[r]&n\ar[r]&n+1\ar@{-->}@/_10pt/[lll]\ar@/^10pt/@{-->}[lll]\\
2\ar@/_30pt/[rrrrrrrru]\\{}
}\]
as desired.

Step 4. We continue with Step 2. Observe that for $4\leq l\leq n-2$ the mutation $\mu^L_l$ takes
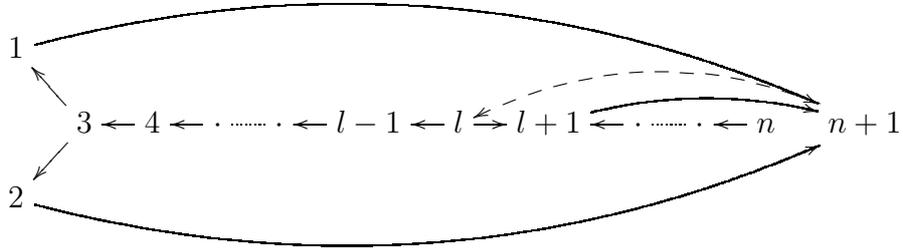
\[\xymatrix@R=1pc@C=1pc{\\ 1\ar@/^30pt/[rrrrrrrrrrrd]\\
&3\ar[lu]\ar[ld]&4\ar[l]&\cdot\ar[l]\ar@{.}[r]&\cdot&l-1\ar[l]&l\ar[l]\ar[r]
&l+1\ar@/^10pt/[rrrr]&\cdot\ar[l]\ar@{.}[r]&\cdot&n\ar[l]&n+1\ar@{-->}@/_20pt/[lllll]\\
2\ar@/_30pt/[rrrrrrrrrrru]\\{}\\{}
}\]
to
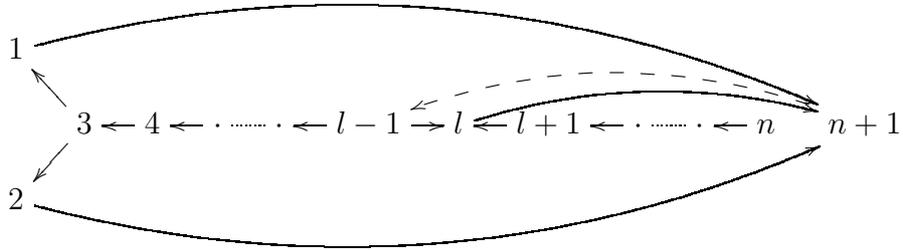
\[\xymatrix@R=1pc@C=1pc{\\ 1\ar@/^30pt/[rrrrrrrrrrrd]\\
&3\ar[lu]\ar[ld]&4\ar[l]&\cdot\ar[l]\ar@{.}[r]&\cdot&l-1\ar[l]\ar[r]&l\ar@/^13pt/[rrrrr]
&l+1\ar[l]&\cdot\ar[l]\ar@{.}[r]&\cdot&n\ar[l]&n+1\ar@{-->}@/_20pt/[llllll]\\
2\ar@/_30pt/[rrrrrrrrrrru]\\{}\\{}
}\]
Therefore
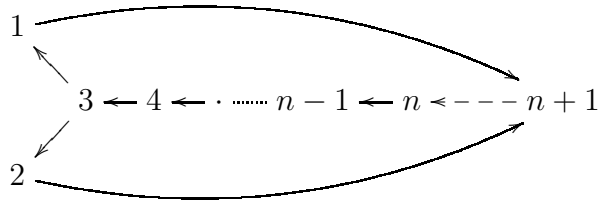
\[\xymatrix@R=1pc@C=1pc{\\ 1\ar@/^20pt/[rrrrrrrd]\\
&3\ar[lu]\ar[ld]&4\ar[l]&\cdot\ar[l]\ar@{.}[r]&n-1&n\ar[l]&&n+1\ar@{-->}[ll]\\
2\ar@/_20pt/[rrrrrrru]\\{} }\]
\[\downarrow\mu^L_n\]
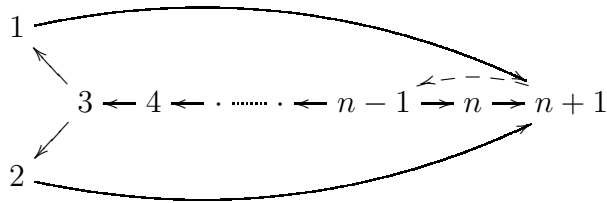
\[\xymatrix@R=1pc@C=1pc{\\ 1\ar@/^20pt/[rrrrrrrd]\\
&3\ar[lu]\ar[ld]&4\ar[l]&\cdot\ar[l]\ar@{.}[r]&\cdot&n-1\ar[l]\ar[r]&n\ar[r]&n+1\ar@/_10pt/@{-->}[ll]\\
2\ar@/_20pt/[rrrrrrru]\\{}
}\]
\[\downarrow\mu^L_{n-1}\]
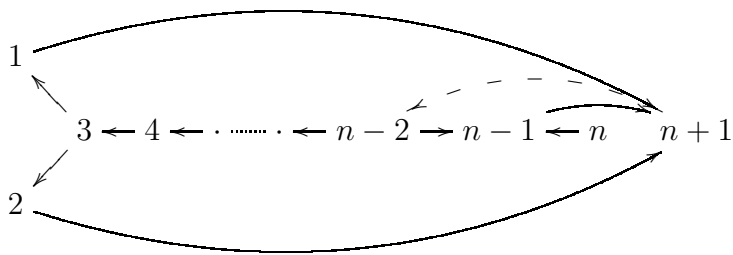
\[\xymatrix@R=1pc@C=1pc{\\ 1\ar@/^30pt/[rrrrrrrrd]\\
&3\ar[lu]\ar[ld]&4\ar[l]&\cdot\ar[l]\ar@{.}[r]&\cdot&n-2\ar[l]\ar[r]&n-1\ar@/^10pt/[rr]&n\ar[l]&n+1\ar@/_20pt/@{-->}[lll]\\
2\ar@/_30pt/[rrrrrrrru]\\{}
}\]
\[\downarrow\mu^L_4\cdots\mu^L_{n-2}\]
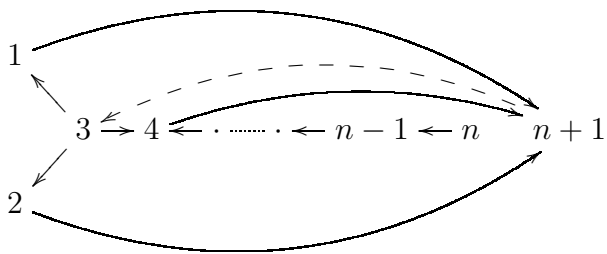
\[\xymatrix@R=1pc@C=1pc{\\ 1\ar@/^30pt/[rrrrrrrd]\\
&3\ar[lu]\ar[ld]\ar[r]&4\ar@/^15pt/[rrrrr]&\cdot\ar[l]\ar@{.}[r]&\cdot&n-1\ar[l]&n\ar[l]&n+1\ar@/_25pt/@{-->}[llllll]\\
2\ar@/_30pt/[rrrrrrru]\\{}
}\]
\[\downarrow\mu^L_3\]
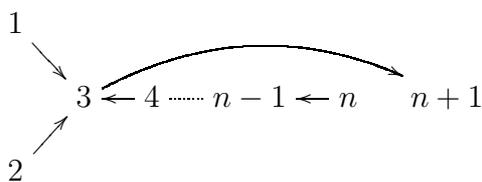
\[\xymatrix@R=1pc@C=1pc{1\ar[dr]\\
&3\ar@/^20pt/[rrrr]&4\ar[l]\ar@{.}[r]&n-1&n\ar[l]&n+1\\
2\ar[ur] }\]
This finishes the proof.
\end{proof}


\newcommand{\cexample}{\noindent {\it Example}~\ref{ex:strong-global-dimension} (Continued).}
\begin{cexample}
Recall that $A$ is the path algebra of the quiver $Q$
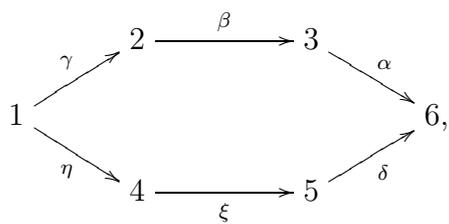
\[\xymatrix@C=1pc@R=1pc{&&2\ar[rrr]^{\beta}&&&3\ar[rrd]^{\alpha}\\
1\ar[rrd]_\eta\ar[rru]^\gamma&&&&&&&6,\\
&&4\ar[rrr]_\xi&&&5\ar[rru]_\delta}\]
modulo the ideal generated by $\beta\gamma$ and $\delta\xi\eta$, and $A'$ is the path algebra of $Q'$
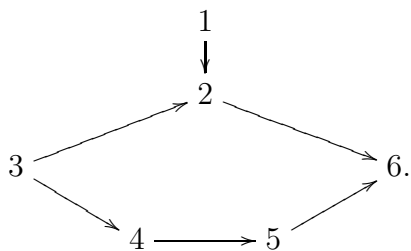
\[\xymatrix@C=1pc@R=1pc{&&&1\ar[d]&&&\\
&&&2\ar[rrrd]&&&\\
3\ar[rrru]\ar[rrd]&&&&&&6.\\
&&4\ar[rr]&&5\ar[rru]&&}\]
Below we establish the derived equivalence between $A$ and $A'$ by using quiver mutation.
According to Theorem~\ref{t:mizuno}, the graded quiver $Q^a$ associated to $A$ is
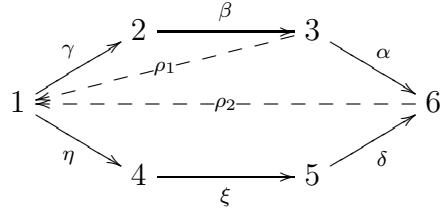
\[\xymatrix@C=1pc@R=1pc{&&2\ar[rrr]^{\beta}&&&3\ar[rrd]^{\alpha}\ar@{-->}[llllld]|{\rho_1}\\
1\ar[rrd]_\eta\ar[rru]^\gamma&&&&&&&6\ar@{-->}[lllllll]|{\rho_2}\\
&&4\ar[rrr]_\xi&&&5\ar[rru]_\delta}\]
which is obtained from $Q$ by adding a dashed arrow in the reverse direction to each relation of $A$ (the solid arrows are in degree $0$ and dashed arrows are in degree $1$), and the associated potential is
$
W^a=\rho_1\beta\gamma+\rho_2\delta\xi\eta
$. This is a rigid quiver with potential. We perform the following sequence of left mutations of graded quivers
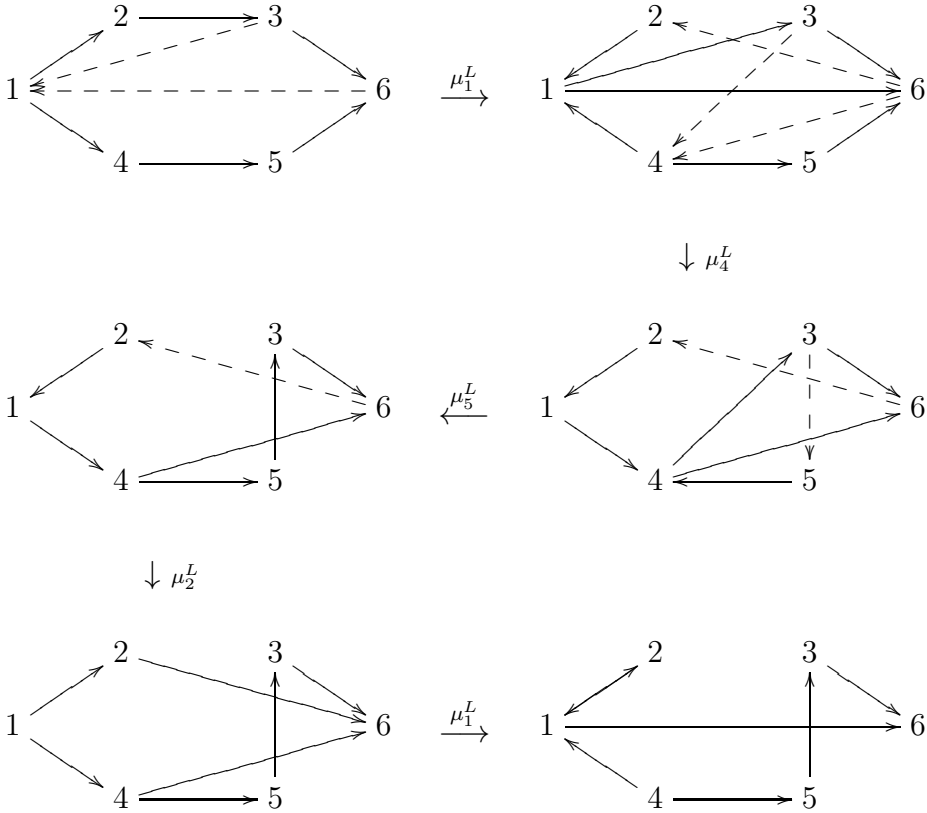
\[
\begin{picture}(400,330)
 \put(0,310){
$\xymatrix@C=0.8pc@R=1pc{&&2\ar[rrr]&&&3\ar[rrd]\ar@{-->}[llllld]\\
1\ar[rrd]\ar[rru]&&&&&&&6\ar@{-->}[lllllll]\\
&&4\ar[rrr]&&&5\ar[rru]}$} 

\put(200,310){
$\xymatrix@C=0.8pc@R=1pc{&&2\ar[lld]&&&3\ar[rrd]\ar@{-->}[llldd]\\
1\ar[rrrrru]\ar[rrrrrrr]&&&&&&&6\ar@{-->}[lllllu]\ar@{-->}[llllld]\\
&&4\ar[llu]\ar[rrr]&&&5\ar[rru]}$}

 \put(0,190){
$\xymatrix@C=0.8pc@R=1pc{&&2\ar[lld]&&&3\ar[rrd]\\
1\ar[rrd]&&&&&&&6\ar@{-->}[lllllu]\\
&&4\ar[rrrrru]\ar[rrr]&&&5\ar[uu]}$} 

\put(200,190){
$\xymatrix@C=0.8pc@R=1pc{&&2\ar[lld]&&&3\ar[rrd]\ar@{-->}[dd]\\
1\ar[rrd]&&&&&&&6\ar@{-->}[lllllu]\\
&&4\ar[rrruu]\ar[rrrrru]&&&5\ar[lll]}$}

 \put(200,70){
$\xymatrix@C=0.8pc@R=1pc{&&2\ar[lld]&&&3\ar[rrd]\\
1\ar[rrrrrrr]\ar[rru]&&&&&&&6\\
&&4\ar[llu]\ar[rrr]&&&5\ar[uu]}$} 

\put(0,70){
$\xymatrix@C=0.8pc@R=1pc{&&2\ar[rrrrrd]&&&3\ar[rrd]\\
1\ar[rrd]\ar[rru]&&&&&&&6\\
&&4\ar[rrrrru]\ar[rrr]&&&5\ar[uu]}$}

\put(170,280){$\stackrel{\scriptstyle{\mu^L_1}}{\longrightarrow}$}
\put(260,220){${\downarrow}$}
\put(270,220){${\scriptstyle{\mu^L_4}}$}
\put(170,160){$\stackrel{\scriptstyle{\mu^L_5}}{\longleftarrow}$}
\put(60,100){$\downarrow$}
\put(70,100){${\scriptstyle{\mu^L_2}}$}
\put(170,40){$\stackrel{\scriptstyle{\mu^L_1}}{\longrightarrow}$}
\end{picture}\]
The last quiver is, up to relabelling of the vertices, the same as the quiver $Q'$. Note that each mutation is taken at a vertex which becomes a source after removing the dashed arrows. Therefore by applying Corollary~\ref{cor:iterated-APR-tilting-via-mutation}, we see that $A$ and $A'$ are derived equivalent. 
\end{cexample}

\bigskip
\newcommand{\autoexample}{\noindent {\it Example}~\ref{ex:auto-silting-type-E} (Continued).}
\begin{autoexample}
Recall that $A=kQ$ and $A'=kQ/(\alpha_4\alpha_2\alpha_1)$, where $Q$ be the quiver (where $n=6,7,8$)
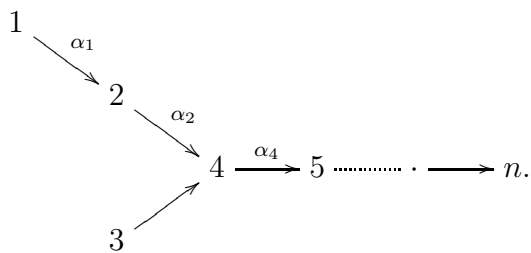
\[
\xymatrix@R=1pc{
1\ar[rd]^{\alpha_1} & \\
& 2 \ar[rd]^{\alpha_2} & \\
& & 4\ar[r]^{\alpha_4} & 5\ar@{.}[r] & \cdot\ar[r] & n. \\
& 3\ar[ru] & 
}
\]
According to Theorem~\ref{t:mizuno}, the graded quiver $Q^a$ associated to $A'$ is
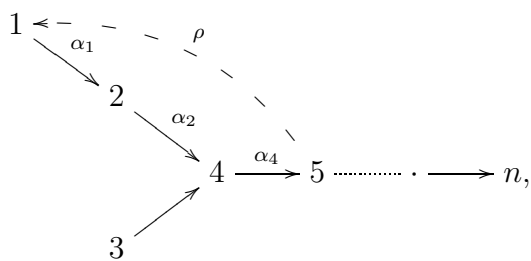
\[
\xymatrix@R=1pc{
1\ar[rd]^{\alpha_1} & \\
& 2 \ar[rd]^{\alpha_2} & \\
& & 4\ar[r]^{\alpha_4} & 5\ar@{.}[r]\ar@{-->}@/_20pt/[uulll]_\rho & \cdot\ar[r] & n, \\
& 3\ar[ru] & 
}
\]
where $\rho$ is in degree $1$ and all other arrows are in degree $0$, 
and the associated potential is $W^a=\rho\alpha_4\alpha_2\alpha_1$. This is a rigid quiver with potential.
By Corollary~\ref{cor:iterated-APR-tilting-via-mutation}, the following sequence of left mutations shows that $A$ and $A'$ are derived equivalent. Note that each mutation is taken at a vertex which becomes a source after removing the dashed arrows.
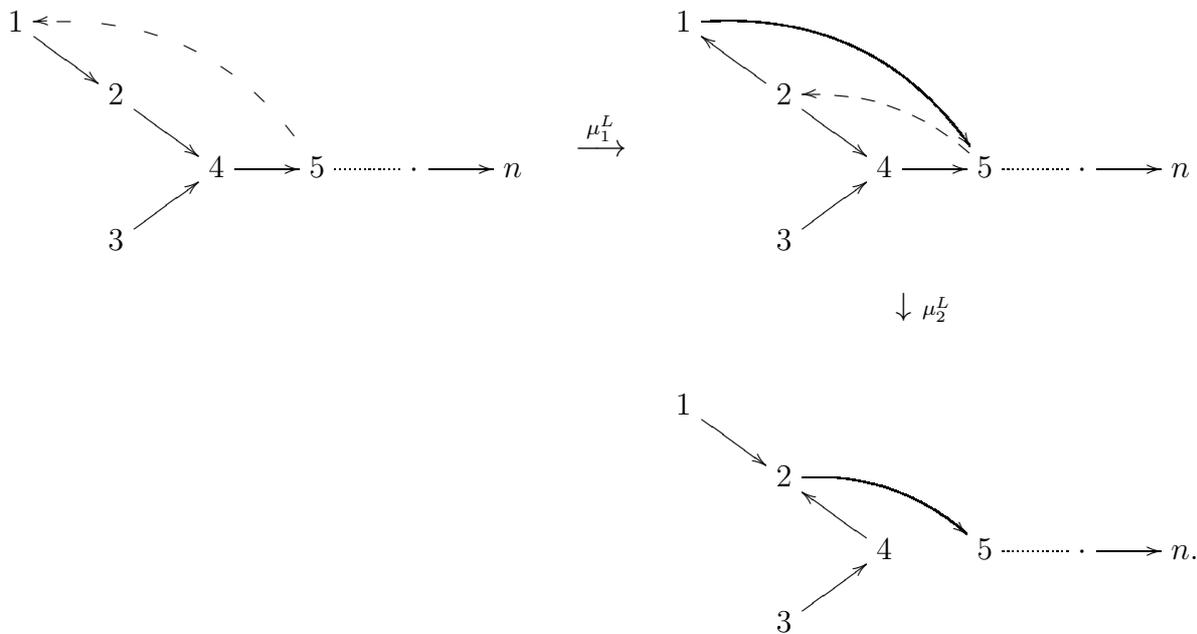
\[
\begin{picture}(600,150)
 \put(-20,130){
\xymatrix@R=1pc{
1\ar[rd] & \\
& 2 \ar[rd] & \\
& & 4\ar[r] & 5\ar@{.}[r]\ar@{-->}@/_20pt/[uulll] & \cdot\ar[r] & n \\
& 3\ar[ru] & 
}} 

\put(230,130){
\xymatrix@R=1pc{
1\ar@/^20pt/[ddrrr] & \\
& 2 \ar[rd] \ar[lu]& \\
& & 4\ar[r] & 5\ar@{.}[r]\ar@{-->}@/_10pt/[ull] & \cdot\ar[r] & n \\
& 3\ar[ru] & 
}}
\put(200,80){$\stackrel{\mu^L_1}{\longrightarrow}$}
\put(320,20){${\downarrow}$}
\put(330,20){${\scriptstyle\mu^L_2}$}
\end{picture}
\]
\[
\begin{picture}(600,100)
\put(230,100){
\xymatrix@R=1pc{
1\ar[dr] & \\
& 2 \ar@/^10pt/[rrd] & \\
& & 4 \ar[lu]& 5\ar@{.}[r] & \cdot\ar[r] & n. \\
& 3\ar[ru] & 
}}
\end{picture}
\]
The last quiver, up to relabelling of the vertices, is $Q$.
\end{autoexample}

\def\cprime{$'$}
\providecommand{\bysame}{\leavevmode\hbox to3em{\hrulefill}\thinspace}
\providecommand{\MR}{\relax\ifhmode\unskip\space\fi MR }
\providecommand{\MRhref}[2]{%
  \href{http://www.ams.org/mathscinet-getitem?mr=#1}{#2}
}
\providecommand{\href}[2]{#2}

\end{document}